\documentclass[12pt]{article}
\usepackage{amssymb,amsmath,amsthm,tikz,multirow}
\usetikzlibrary{arrows,calc}

\title{Pentagonal Subdivision}
\author{Min Yan\thanks{Research was supported by Hong Kong RGC General Research Fund 16303515.} \\ Hong Kong University of Science and Technology}

\newcommand*\circled[1]{\tikz[baseline=(char.base)]{
             \node[shape=circle,draw,inner sep=0.5pt] (char) {#1};}}

\newtheorem{theorem}{Theorem}
\newtheorem*{theorem*}{Theorem}

\newtheorem{proposition}[theorem]{Proposition}

\theoremstyle{definition}
\newtheorem*{definition*}{Definition}
\newtheorem*{case*}{Case}
\newtheorem*{subcase*}{Subcase}

\theoremstyle{remark}

\numberwithin{equation}{section}

\begin{document}

\maketitle

\begin{abstract}
We develop a theory of simple pentagonal subdivision of quadrilateral tilings, on orientable as well as non-orientable surfaces. Then we apply the theory to answer questions related to pentagonal tilings of surfaces, especially those related to pentagonal or double pentagonal subdivisions. 
\end{abstract}

\section{Introduction}

Pentagonal subdivision appeared in \cite{wy1}, as a way of producing tilings of oriented surface by pentagons. When this process is applied to platonic solids, we get three families of edge-to-edge tilings of the sphere by congruent pentagons. We proved in \cite{wy1} that these are the only tilings of the sphere by congruent pentagons with edge length combination $a^2b^2c$. When some among $a,b,c$ become equal, the reduced form of the tilings also appeared in \cite{ay} and \cite{wy2}. A natural combinatorial question is the criterion for a pentagonal tiling to be a pentagonal subdivision. 

Pentagonal subdivision is also related to the problem of distribution of vertices of degree $>3$, which we call {\em high degree vertices}, reflecting negative curvature in a pentagonal tiling of the sphere. This problem is motivated by the fact that most vertices in a pentagonal tiling of the sphere have degree $3$. In \cite{yan}, we studied the extreme case of few high degree vertices, which means high concentration of high degree. Specifically, we proved that the tiling cannot have only one high degree vertex. Moreover, if the number is two, then the tiling is the explicit earth map tiling. At the other extreme, we are interested in the case that each tile has exactly one high degree vertex, which means evenly distributed high degree. Any pentagonal subdivision of a triangular tiling in which all vertices have degree $>3$ is such an example. We speculated that the converse is also true.

In \cite{wy1}, we also introduced double pentagonal subdivision, and proved that tilings of the sphere by congruent pentagons with edge length combination $a^3bc$ are two pentagonal subdivisions of platonic solids. A natural related question is a criterion for a tiling to admit double pentagonal subdivision. 

We answer these questions in the present paper, specifically through Theorems \ref{ps_thm1}, \ref{ps_thm2}, and \ref{dps_thm2}. In pursuing these, we discover a more fundamental construction, that is the simple pentagonal subdivision that turns some quadrilateral tilings into pentagonal tilings. In Section \ref{quad}, we develop this theory over any surface. We find criteria for a quadrilateral tiling to admit simple pentagonal subdivision. For example, Proposition \ref{quad_theroem3} says that on orientable surfaces, the condition is exactly the bipartite property. We further study degenerate quadrilaterals appearing in tilings admitting simple pentagonal subdivisions. Then we get very explicit criteria for some surfaces. For example, Theorem \ref{sphere} says that a quadrilateral tiling of the sphere can only have the non-degenerate tile or the specific degenerate tile $Q_{13}$. Moreover, any quadrilateral tiling of the sphere admits a simple pentagonal subdivision. 

In Section 3, we apply the theory in Section 2 to pentagonal tilings. In addition to answering our original questions, we also answer questions related to simple pentagonal subdivision. Theorems \ref{sp_thm1} and \ref{sp_thm2} give a criterion for a pentagonal tiling to come from simple pentagonal subdivision. As part of the discussion on double pentagonal subdivision, we also introduce quadrilateral subdivision in Section \ref{double}. Then in Theorem \ref{dps_thm1}, we give a criterion for a quadrilateral tiling to come from quadrilateral subdivision.

Finally, we fix the exact context of this paper. The surfaces are compact, connected and without boundary. By tiling, we mean a graph embedded in the surface, so that the surface is divided into open regions (i.e., tiles) homeomorphic to the open $2$-disk. This means that tilings are edge-to-edge, the degree of each vertex is at least $3$, and each tile has at least $3$ edges.

\section{Subdivisible Quadrilateral Tiling}
\label{quad}

Given a quadrilateral tiling, we would like to divide each tile into two halves, by connecting the middle points of an opposite pair of edges in each tile; we furthermore require that the middle point of each edge is used exactly once. Then each quadrilateral tile is divided into two pentagonal tiles, and we get a pentagonal tiling. See Figure \ref{4subdivision}, where the subdivision of quadrilaterals are indicated by dotted lines. We call the pentagonal tiling the {\em simple pentagonal subdivision} of the quadrilateral tiling. Not all  quadrilateral tilings can admit simple pentagonal subdivision.

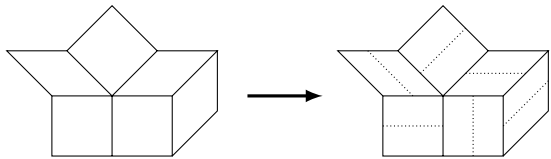
\begin{figure}[ht]
\centering
\begin{tikzpicture}[>=latex]

\foreach \a in {-1,0}
\draw[xshift=4.4*\a cm]
	(-0.8,-0.8) rectangle (0.8,0)
	(0,-0.8) -- (0,0)
	(-0.8,0) -- (-1.4,0.6) -- (-0.6,0.6) -- (0,0)
	(0.8,0) -- (1.4,0.6) -- (0.6,0.6) -- (0,0)
	(-0.6,0.6) -- (0,1.2) -- (0.6,0.6)
	(1.4,0.6) -- (1.4,-0.2) -- (0.8,-0.8);

\draw[very thick,->]
	(-2.6,0) -- ++(1,0);

\draw[densely dotted]
	(-0.8,-0.4) -- (0,-0.4)
	(0.4,-0.8) -- (0.4,0)
	(-0.4,0) -- (-1,0.6)
	(-0.3,0.3) -- (0.3,0.9)
	(0.3,0.3) -- (1.1,0.3)
	(0.8,-0.4) -- (1.4,0.2);

\end{tikzpicture}
\caption{Simple pentagonal subdivision.}
\label{4subdivision}
\end{figure}

\begin{definition*}
A quadrilateral tiling of a surface is {\em (pentagonally) subdivisible} if it admits a simple pentagonal subdivision.
\end{definition*}

We allow tiles to be degenerate. Here is the definition. 

\begin{definition*}
A polygon is {\em non-degenerate} if its boundary is a simple closed curve. Otherwise the polygon is {\em degenerate}.
\end{definition*}

Non-degenerate means no identification of vertices, and degenerate means that some vertices or even edges are identified. The identification of edges must be in pairs in order to get a surface. Figure \ref{degenerate} shows that the identified edge pair can be opposing or twisted, with the only exception that adjacent opposing edges cannot be identified because it produces a vertex of degree $1$.

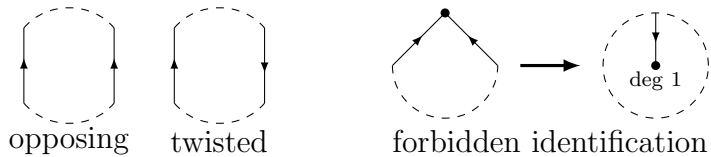
\begin{figure}[h]
\centering
\begin{tikzpicture}[>=latex,scale=1]


\foreach \a in {0,1}
{
\begin{scope}[xshift=2*\a cm]

\draw
	(-0.6,-0.5) -- (-0.6,0.5)
	(0.6,-0.5) -- (0.6,0.5);

\draw[dashed]
	(-0.6,0.5) to[out=50, in=130] (0.6,0.5)
	(-0.6,-0.5) to[out=-50, in=-130] (0.6,-0.5);

\draw[->]
	(-0.6,0) -- ++(0,0.1);

\end{scope}
}

\draw[->]
	(0.6,0) -- ++(0,0.1);
	
\node at (0,-1) {opposing};

\draw[->]
	(2.6,0) -- ++(0,-0.1);

\node at (2,-1) {twisted};


\begin{scope}[xshift=5cm]

\draw
	(-0.7,0) -- (0,0.7) -- (0.7,0);

\draw[dashed]
	(-0.7,0) to[out=-90, in=180] 
	(0,-0.7) to[out=0, in=-90] 
	(0.7,0);

\fill 
	(0,0.7) circle (0.06);
	
\draw[->]
	(-0.4,0.3) -- ++(0.1,0.1);
\draw[->]
	(0.4,0.3) -- ++(-0.1,0.1);

\draw[->, very thick]
	(1,0) -- ++(0.8,0);

\node at (1.4,-1) {forbidden identification};
		
\end{scope}

\begin{scope}[xshift=7.8cm]

\draw[dashed]
	(0,0) circle (0.7);

\draw
	(0,0.7) -- (0,0);

\fill 
	(0,0) circle (0.06);
	
\draw[->]
	(0,0.4) -- ++(0,-0.1);

\node at (0,-0.2) {\scriptsize deg 1};
		
\end{scope}

\end{tikzpicture}
\caption{Degeneracy by identifying edges.}
\label{degenerate}
\end{figure}

\subsection{Criterion for Subdivisibility}
\label{criterion}

Figure \ref{subdiv_example} gives a subdivisible tiling $T_4$ of a $4$-gon, a subdivisible tiling $T_2$ of a $2$-gon, and a subdivisible tiling of a torus $T^2$. By regarding the outside of $T_4$ as one quadrilateral, we see that the cube tiling of the sphere $S^2$ is subdivisible. If we identify antipidal points on the boundary of the $2$-gon, then we get a subdivisible tiling of projective space $P^2$, with two degenerate tiles.

\begin{figure}[h]
\centering
\begin{tikzpicture}[>=latex,scale=1]


\foreach \a in {0,...,3}
\draw[rotate=90*\a]
	(0.8,-0.8) -- (0.8,0.8) -- (0.3,0.3) -- (0.3,-0.3)
	;	

\draw[densely dotted]
	(-0.3,0) -- (0.3,0)
	(0,0.8) -- (0,0.3)
	(0,-0.8) -- (0,-0.3)
	(0.55,0.55) -- (0.55,-0.55)
	(-0.55,0.55) -- (-0.55,-0.55);

\node at (1.1,-0.8) {$T_4$};	


\begin{scope}[xshift=2.7cm]

\draw
	(0,0) circle (0.9)
	(0,-0.9) -- (0.4,-0.4) -- (0.4,0.4) -- (0,0.9)
	(0,-0.9) -- (-0.4,-0.4) -- (-0.4,0.4) -- (0,0.9)
	(0.4,-0.4) -- (0,-0.3) -- (-0.4,-0.4)
	(0.4,0.4) -- (0,0.3) -- (-0.4,0.4)
	(0,-0.3) -- (0,0.3);
	
\draw[densely dotted]
	(-0.9,0) -- (-0.4,0)
	(-0.2,-0.65) -- (0.2,-0.35)
	(-0.2,0.65) -- (0.2,0.35)
	(-0.2,-0.35) -- (-0.2,0.35)
	(0,0) -- (0.4,0)
	(0.2,-0.65) to[out=0,in=-90] 
	(0.7,0) to[out=90,in=0] 
	(0.2,0.65);

\node at (0.9,-0.8) {$T_2$};
		
\end{scope}


\begin{scope}[xshift=5.6cm]

\foreach \a in {0,...,3}
{
\draw[rotate=90*\a]
	(0,0) -- (0.8,0) -- (0.8,0.8) -- (0,0.8);
\draw[densely dotted,rotate=90*\a]
	(0.4,0) -- (0.4,0.8);
}

\foreach \a in {-1,1}
\foreach \b in {-1,1}
{
\draw[->] (0.4*\a,0.8*\b) -- ++(0.1,0);
\draw[->] (0.8*\a,0.4*\b) -- ++(0,0.1);	
}

\foreach \a in {-1,1}
{
\node at (0.4,\a) {\scriptsize $a_0$};
\node at (-0.4,\a) {\scriptsize $a_1$};
\node at (\a,0.4) {\scriptsize $b_0$};
\node at (\a,-0.4) {\scriptsize $b_1$};
}
	
\end{scope}

\end{tikzpicture}
\caption{Subdivisible tilings of $4$-gon, $2$-gon, and torus.}
\label{subdiv_example}
\end{figure}
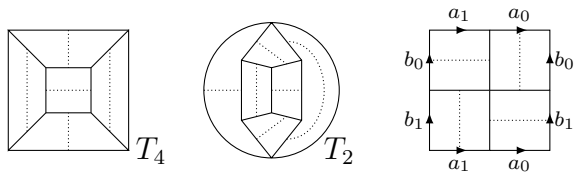

Figure \ref{operation} shows three operations on simple pentagonal subdivisions. The first is {\em dual subdivision}. This means that, for each tile, we switch the connection of middle points of one opposite pair of edges to the other opposite pair of edges. The second is {\em refinement}, given by applying a $3\times 3$ grid to each tile. We may also use $5\times 5$ grid, and so on. The refinement can be used to turn degenerate tiles into non-degenerate tiles. The third is {\em connected sum} of simple pentagonal subdivisions of two surfaces along non-degenerate tiles. Specifically, we delete one non-degenerate tile from each surface, and then glue along the boundaries of the deleted tiles in a twisted way. 

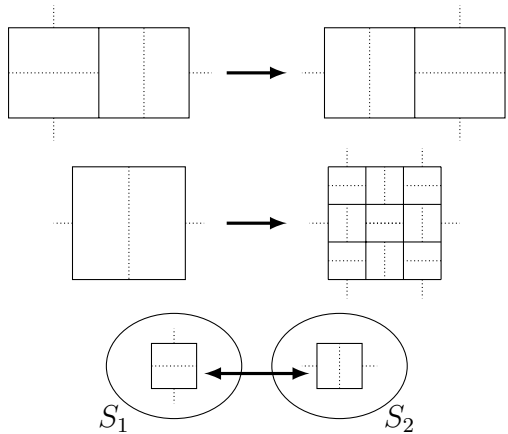
\begin{figure}[h]
\centering
\begin{tikzpicture}[>=latex,scale=1]

\foreach \a in {0,1}
\draw[xshift=4.2*\a cm]
	(-1.2,-0.6) rectangle (1.2,0.6)
	(0,-0.6) -- (0,0.6);

\draw[densely dotted]
	(-1.2,0) -- (0,0)
	(0.6,-0.6) -- (0.6,0.6)
	(-0.6,0.6) -- ++(0,0.3)	
	(-0.6,-0.6) -- ++(0,-0.3)
	(1.2,0) -- ++(0.3,0);
	
\draw[->, very thick]
	(1.7,0) -- ++(0.8,0);

\draw[xshift=4.2 cm, densely dotted]
	(1.2,0) -- (0,0)
	(-0.6,-0.6) -- (-0.6,0.6)
	(0.6,0.6) -- ++(0,0.3)	
	(0.6,-0.6) -- ++(0,-0.3)
	(-1.2,0) -- ++(-0.3,0);


\begin{scope}[shift={(0.4cm,-2cm)}]

\draw
	(-0.75,-0.75) rectangle (0.75,0.75);

\draw[densely dotted]
	(0,0.75) -- (0,-0.75)
	(0.75,0) -- (1,0)
	(-0.75,0) -- (-1,0);

\draw[very thick, ->]
	(1.3,0) -- ++(0.8,0);

\begin{scope}[xshift=3.4cm]

\foreach \a in {-3,-1,1,3}
\draw
	(-0.75,0.25*\a) -- (0.75,0.25*\a)
	(0.25*\a,-0.75) -- (0.25*\a,0.75);

\foreach \a in {1,-1}
\draw[densely dotted]
	(-0.75,0.5*\a) -- ++(0.5,0)
	(0.75,0.5*\a) -- ++(-0.5,0)
	(-0.25,0) -- (0.25,0)
	(0,0.25*\a) -- ++(0,0.5*\a)
	(0.5*\a,-0.25) -- ++(0,0.5)
	(0.75*\a,0) -- ++(0.25*\a,0)
	(0.5*\a,0.75) -- ++(0,0.25)
	(0.5*\a,-0.75) -- ++(0,-0.25);

\end{scope}

\end{scope}


\begin{scope}[shift={(2.1cm,-3.9cm)}]

\foreach \a in {-1,1}
{
\begin{scope}[xshift=1.1*\a cm]

\draw
	(0,0) ellipse (0.9 and 0.7);

\draw
	(-0.3,-0.3) rectangle (0.3,0.3);

\end{scope}
}

\draw[xshift=-1.1 cm, densely dotted]
	(-0.3,0) -- (0.3,0)
	(0,-0.3) -- ++(0,-0.2)
	(0,0.3) -- ++(0,0.2);

\draw[xshift=1.1 cm, densely dotted]
	(0,-0.3) -- (0,0.3)
	(-0.3,0) -- ++(-0.2,0)
	(0.3,0) -- ++(0.2,0);
	
\draw[very thick, <->]
	(-0.7,-0.1) -- (0.7,-0.1);
	
\node at (-1.9,-0.7) {$S_1$};
\node at (1.9,-0.7) {$S_2$};

\end{scope}

\end{tikzpicture}
\caption{Dual subdivision, refinement, and connected sum.}
\label{operation}
\end{figure}

\begin{proposition}\label{qonly}
Any surface has a subdivisible tiling in which all tiles are non-degenerate.
\end{proposition}

\begin{proof}
Figure \ref{subdiv_example} gives subdivisible tilings of $S^2,P^2,T^2$. The tiles for $S^2,T^2$ are non-degenerate. We may apply refinement to get a subdivisible tiling of $P^2$ consisting of non-degenerate tiles only. Then we may use connected sum to construct subdivisible tilings of all other surfaces with only non-degenerate tiles.
\end{proof}

A simple pentagonal subdivision is equivalent to assigning an orthogonal direction at the middle of each edge, such that two edges adjacent at a vertex have compatible assignments. See the first of Figure \ref{vertex_orient}. We call the assignment {\em subdivision of edge}. All the compatible subdivisions of all edges at a vertex give an orientation of the vertex. The dual subdivision simply reverses subdivisions of all edges (changing the orthogonal directions to the opposites). This reverses all the orientations of vertices.

\begin{figure}[h]
\centering
\begin{tikzpicture}[>=latex,scale=1]


\begin{scope}[shift={(-5cm,-0.5cm)}]

\draw
	(40:1.6) -- (0,0) -- (1.6,0);

\draw[densely dotted]
	(40:0.8) -- ++(130:0.3)
	(0.8,0) -- ++(0,0.3);

\draw[->]
	(-40:0.4) arc (-40:180:0.4);
	
\fill
	(0,0) circle (0.06);

\node at (0:1.8) {\small 1};
\node at (40:1.8) {\small 2};

\end{scope}


\foreach \a in {1,...,7}
{
\draw[xshift=-1 cm]
	(0,0) -- (45*\a:1);
\draw[xshift=-1 cm, densely dotted]
	(45*\a:0.7) -- ++(45*\a+90:0.3);

\draw[xshift=1 cm]
	(0,0) -- (180+45*\a:1);
\draw[xshift=1 cm, densely dotted]
	(180+45*\a:0.7) -- ++(45*\a+90:0.3);
}

\draw
	(-1,0) -- (1,0);
\draw[densely dotted]
	(0,0) -- ++(0,0.3);

\fill
	(-1,0) circle (0.06);
\fill
	(1,0) circle (0.06);

\draw[->]
	(-1,-0.4) arc (-90:210:0.4);
\draw[->]
	(1,-0.4) arc (270:-30:0.4);

\end{tikzpicture}
\caption{Subdivision of edge, and orientation at vertex.}
\label{vertex_orient}
\end{figure}
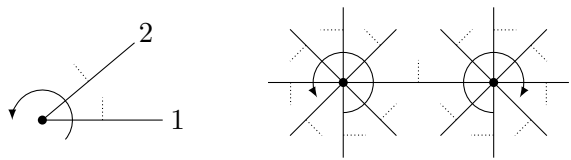

The second of Figure \ref{vertex_orient} shows that two vertices at the end of any edge should have opposite orientations.  This characterises subdivisible tilings.

\begin{proposition}\label{quad_theroem1}
A quadrilateral tiling of a surface is subdivisible if and only if it is possible to assign orientations at all vertices, such that the orientations at the two ends of each edge are opposite to each other. 
\end{proposition}

\begin{proof}
We only need to argue that an orientation assignment satisfying the proposition induces simple pentagonal subdivision. First, the opposite orientations at the two ends of an edge determine the subdivision of the edge. Then in a quadrilateral in the first of Figure \ref{orient2subdiv}, the opposite orientation property implies that subdivisions of the four edges induce a subdivision of the quadrilateral. We note that the process only makes use of the orientation inside the tile, and therefore includes the possibility that the quadrilateral tile is degenerate, such as the second of Figure \ref{orient2subdiv}, in which a pair of vertices of a quadrilateral tile are identified.
\end{proof}

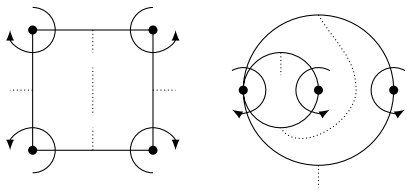
\begin{figure}[h]
\centering
\begin{tikzpicture}[>=latex,scale=1]


\draw
	(-0.8,-0.8) rectangle (0.8,0.8);

\draw[densely dotted]
	(0,-0.8) -- ++(0,0.3)
	(0,0.8) -- ++(0,-0.3)
	(0.8,0) -- ++(0.3,0)
	(-0.8,0) -- ++(-0.3,0)
	(0,-0.3) -- (0,0.3);	

\foreach \a in {0,...,3}
\fill[rotate=90*\a]
	(0.8,0.8) circle (0.06);

\foreach \a in {1,-1}
{
\draw[scale=\a,->]
	(0.8,1.1) arc (90:360:0.3);
\draw[scale=\a,->]
	(-0.8,1.1) arc (90:-180:0.3);
}


\begin{scope}[xshift=3cm]

\draw
	(-0.5,0) circle (0.5)
	(0,0) circle (1);

\draw[densely dotted]
	(0,1) to[out=-60,in=90] 
	(0.5,0) to[out=-90,in=-60]  
	(-0.5,-0.5)
	(-0.5,0.5) -- ++(0,-0.3)
	(0,-1) -- ++(0,-0.3);	
	
\fill
	(-1,0) circle (0.06)
	(0,0) circle (0.06)
	(1,0) circle (0.06);

\draw[->,xshift=-1 cm]
	(120:0.3) arc (120:-120:0.3);
	
\foreach \a in {0,1}
\draw[->,xshift=\a cm]
	(60:0.3) arc (60:300:0.3);
	
\end{scope}

\end{tikzpicture}
\caption{Vertex orientations determine subdivision.}
\label{orient2subdiv}
\end{figure}

For a quadrilateral tiling of a surface $S$, we may start with an orientation $o_v$ at a vertex $v$. We connect to any other vertex $w$ by a path $v=v_0-v_1-\cdots-v_k=w$ (each ``$-$'' is one edge), and move $o_v$ to an orientation $o_w$ along the path. Then we assign $o_w$ to $w$ for even $k$, and assign $-o_w$ for odd $k$. The problem is what happens when $w=v$, i.e., the path is a cycle. This requires the assignment to be compatible with the orientation homomorphism $\pi_1S\to \{\pm 1\}$. If $S=kT^2$ (connected sum of $k$ copies of tori) is an orientable surface, then the fundamental group $\pi_1S$ is generated by {\em fundamental cycles} $a_1,b_1,\dots,a_k,b_k$, subject to the relation $a_1b_1a_1^{-1}b_1^{-1}\cdots a_kb_ka_k^{-1}b_k^{-1}=1$. If $S=kP^2$ (connected sum of $k$ copies of projective space) is a non-orientable surface, then $\pi_1S$ is generated by {\em fundamental cycles} $a_1,\dots,a_k$, subject to the relation $a_1^2\cdots a_k^2=1$. In case $S$ is the sphere, the fundamental group is trivial, and we do not need to be concerned with fundamental cycles.

\begin{proposition}\label{quad_theroem2}
A quadrilateral tiling of a surface is subdivisible if and only if the fundamental cycles of the surface can be represented by cycles in the tiling in the following way:
\begin{enumerate}
\item For an orientable surface, the fundamental cycles are represented by cycles with an even number of edges. 
\item For a non-orientable surface, the fundamental cycles are represented by cycles with an odd number of edges. 
\end{enumerate}
\end{proposition}

\begin{proof}
For the orientation assignment to be compatible with the orientation homomorphism $\pi_1S\to \{\pm 1\}$, we must have the compatibility for generators. For an orientable surface, the orientation homomorphism takes all fundamental cycles $a_i,b_i$ to $1$. This means even number of edges in the first part. For a non-orientable surface, the orientation homomorphism takes all fundamental cycles $a_i$ to $-1$. This means odd number of edges in the second part. 

Conversely, suppose we have the compatibility for fundamental cycles. Then we have the compatibility for combinations of fundamental cycles. Now any cycle is homotopic to a combination of fundamental cycles. The homotopy can be reduced to the replacements of some consecutive edges of a tile by the remaining edges of the same tile. For a quadrilateral tile, this means $i$ edges replaced by $4-i$ edges. Such replacement preserves the parity of the number of edges in a cycle. Therefore we have the compatibility for all cycles.
\end{proof}

The interpretation of subdivisibility in terms of the parities of the numbers of edges in cycles implies the following.

\begin{proposition}\label{quad_theroem3}
A quadrilateral tiling on an orientable surface is subdivisible if and only if the underlying graph is bipartite. The underlying graph of a subdivisible quadrilateral tiling on a non-orientable surface cannot be bipartite. 
\end{proposition}

\subsection{Degeneracy of Tiles}

The criteria in Section \ref{criterion} are qualitative. We may get more specific criteria by studying how degenerate a tile can be in a subdivisible quadrilateral tiling. 

Figure \ref{quad_degenerate} shows all quadrilateral tiles. Some quadrilaterals are not labelled because we will show that they cannot be tiles in subdivisible tilings. The first row consists of the non-degenerate quadrilateral $Q$, and degenerate quadrilateral tiles obtained by identifying vertices. We label the four vertices of $Q$ by $1,2,3,4$, and the vertices are not identified. We use the same labels for the vertices of all quadrilaterals. For example, $Q_{12}$ means that the vertices $1,2$ are identified, and $Q_{12,34}$ means that $1,2$ are identified and $3,4$ are also identified. 

We need to pay attention to how the tiles embed into surfaces. This means taking the union of a quadrilateral tile with a disk neighborhood of the identified vertex. Different configurations of the disk neighborhood may yield different embeddings. This happens when $1,2,3$ are identified, and there are two possible embeddings $Q_{123},Q_{132}$. A similar phenomenon occurs when $1,2,3,4$ are identified. 

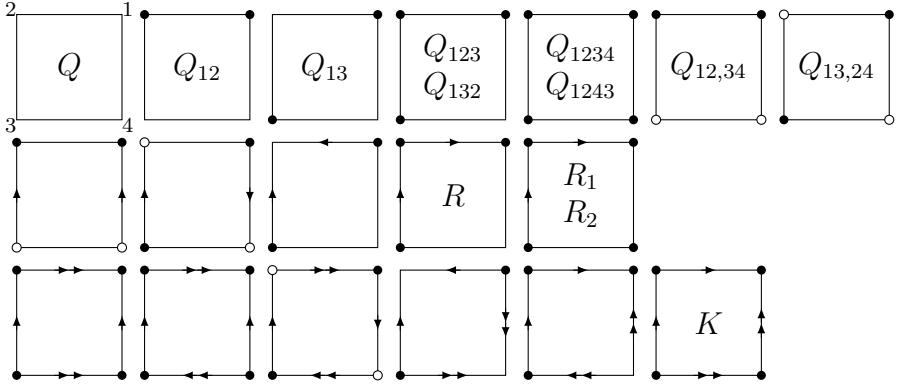
\begin{figure}[ht]
\centering
\begin{tikzpicture}[>=latex,scale=1]


\foreach \a in {0,...,6}
\draw[xshift=1.7*\a cm]
	(-0.7,-0.7) rectangle (0.7,0.7);

\foreach \a in {1,...,4}
\node at (-45+90*\a:1.1) {\scriptsize \a};

\node at (0,0) {$Q$};
\node at (1.7,0) {$Q_{12}$};
\node at (3.4,0) {$Q_{13}$};
\node at (5.1,0.25) {$Q_{123}$};
\node at (5.1,-0.25) {$Q_{132}$};
\node at (6.8,0.25) {$Q_{1234}$};
\node at (6.8,-0.25) {$Q_{1243}$};
\node at (8.5,0) {$Q_{12,34}$};
\node at (10.2,0) {$Q_{13,24}$};

\foreach \a in {1,2,3,4,5,6}
\fill[xshift=1.7*\a cm] (0.7,0.7) circle (0.06);
\foreach \a in {1,3,4,5}
\fill[xshift=1.7*\a cm] (-0.7,0.7) circle (0.06);
\foreach \a in {2,3,4,6}
\fill[xshift=1.7*\a cm] (-0.7,-0.7) circle (0.06);
\foreach \a in {4}
\fill[xshift=1.7*\a cm] (0.7,-0.7) circle (0.06);
\filldraw[fill=white, xshift=1.7*5 cm] 
	(-0.7,-0.7) circle (0.06)
	(0.7,-0.7) circle (0.06);
\filldraw[fill=white, xshift=1.7*6 cm] 
	(-0.7,0.7) circle (0.06)
	(0.7,-0.7) circle (0.06);


\begin{scope}[yshift=-1.7cm]

\foreach \a in {0,...,4}
\draw[xshift=1.7*\a cm]
	(-0.7,-0.7) rectangle (0.7,0.7);

\node at (1.7*3,0) {$R$};
\node at (1.7*4,0.25) {$R_1$};
\node at (1.7*4,-0.25) {$R_2$};

\foreach \a in {3,4}
\draw[->, xshift=1.7*\a cm] (-0.1,0.7) -- ++(0.2,0);

\draw[<-] (-0.1+1.7*2,0.7) -- ++(0.2,0);

\foreach \a in {0,...,4}
\draw[->, xshift=1.7*\a cm] (-0.7,-0.1) -- ++(0,0.2);

\draw[->] (0.7,-0.1) -- ++(0,0.2);

\draw[->] (0.7+1.7*1,0.1) -- ++(0,-0.2);

\foreach \a in {0,1,2,3,4}
\fill[xshift=1.7*\a cm] (0.7,0.7) circle (0.06);
\foreach \a in {0,3,4}
\fill[xshift=1.7*\a cm] (-0.7,0.7) circle (0.06);
\foreach \a in {1,2,3,4}
\fill[xshift=1.7*\a cm] (-0.7,-0.7) circle (0.06);
\foreach \a in {4}
\fill[xshift=1.7*\a cm] (0.7,-0.7) circle (0.06);
\foreach \a in {1}
\filldraw[xshift=1.7*\a cm,fill=white] (-0.7,0.7) circle (0.06);
\foreach \a in {0}
\filldraw[xshift=1.7*\a cm,fill=white] (-0.7,-0.7) circle (0.06);
\foreach \a in {0,1}
\filldraw[xshift=1.7*\a cm,fill=white] (0.7,-0.7) circle (0.06);

\end{scope}


\begin{scope}[yshift=-3.4cm]

\foreach \a in {0,...,5}
\draw[xshift=1.7*\a cm]
	(-0.7,-0.7) rectangle (0.7,0.7);

\node at (1.7*5,0) {$K$};

\foreach \a in {0,1,2}
{
\draw[->, xshift=1.7*\a cm] (0,0.7) -- ++(0.2,0);
\draw[->, xshift=1.7*\a cm] (-0.2,0.7) -- ++(0.2,0);
}

\draw[->] (0.1+1.7*3,0.7) -- ++(-0.2,0);

\foreach \a in {4,5}
\draw[->] (-0.1+1.7*\a,0.7) -- ++(0.2,0);

\foreach \a in {0,3,5}
{
\draw[->, xshift=1.7*\a cm] (0,-0.7) -- ++(0.2,0);
\draw[->, xshift=1.7*\a cm] (-0.2,-0.7) -- ++(0.2,0);
}

\foreach \a in {1,2,4}
{
\draw[->, xshift=1.7*\a cm] (0,-0.7) -- ++(-0.2,0);
\draw[->, xshift=1.7*\a cm] (0.2,-0.7) -- ++(-0.2,0);
}

\foreach \a in {0,...,5}
\draw[->, xshift=1.7*\a cm] (-0.7,-0.1) -- ++(0,0.2);

\foreach \a in {0,1}
\draw[->, xshift=1.7*\a cm] (0.7,-0.1) -- ++(0,0.2);

\draw[->, xshift=1.7*2 cm] (0.7,0.1) -- ++(0,-0.2);

\foreach \a in {4,5}
{
\draw[->, xshift=1.7*\a cm] (0.7,0) -- ++(0,0.2);
\draw[->, xshift=1.7*\a cm] (0.7,-0.2) -- ++(0,0.2);
}

\draw[->, xshift=1.7*3 cm] (0.7,0) -- ++(0,-0.2);
\draw[->, xshift=1.7*3 cm] (0.7,0.2) -- ++(0,-0.2);

\foreach \a in {0,1,2,3,4,5}
\fill[xshift=1.7*\a cm] (0.7,0.7) circle (0.06);
\foreach \a in {0,1,4,5}
\fill[xshift=1.7*\a cm] (-0.7,0.7) circle (0.06);
\foreach \a in {0,1,2,3,4,5}
\fill[xshift=1.7*\a cm] (-0.7,-0.7) circle (0.06);
\foreach \a in {0,1,5}
\fill[xshift=1.7*\a cm] (0.7,-0.7) circle (0.06);
\foreach \a in {2}
\filldraw[xshift=1.7*\a cm,fill=white] (-0.7,0.7) circle (0.06);
\foreach \a in {2}
\filldraw[xshift=1.7*\a cm,fill=white] (0.7,-0.7) circle (0.06);

\end{scope}

\end{tikzpicture}
\caption{Quadrilateral tiles.}
\label{quad_degenerate}
\end{figure}

The second row consists of degenerate quadrilateral tiles obtained by identifying one pair of edges. The identification of edges also induces identification of vertices. The first of the second row is a cylinder. In the first of Figure \ref{quad_degenerate1}, we indicate the opposite orientations at vertices 1 and 2 implied by Proposition \ref{quad_theroem1}. In the second picture, we glue the two edges to get a cylinder, and find that the two orientations are still opposite at the identified vertex. Therefore the quadrilateral cannot be a tile in subdivisible tiling, and cannot be used. The second of the second row is a M\"obius band, and can be dismissed for the same reason. The third quadrilateral is dismissed because adjacent opposing pair of edges are identified, and such identification is forbidden by Figure \ref{degenerate}.

\begin{figure}[ht]
\centering
\begin{tikzpicture}[>=latex,scale=1]


\draw
	(-0.7,-0.7) rectangle (0.7,0.7);

\fill
	(-0.7,0.7) circle (0.06)
	(0.7,0.7) circle (0.06);
	
\draw[->]
	(-0.7,-0.1) -- ++(0,0.2);
\draw[->]
	(0.7,-0.1) -- ++(0,0.2);	

\draw[shift={((0.7 cm,0.7 cm))},->]
	(-100:0.4) arc (-100:-170:0.4);
\draw[shift={((-0.7 cm,0.7 cm))},->]
	(-80:0.4) arc (-80:-10:0.4);

\node at (0.6,0.85) {\scriptsize 1};
\node at (-0.6,0.85) {\scriptsize 2};

\begin{scope}[xshift=2.4cm]

\draw
	(0.7,0.6) arc (0:360:0.7 and 0.2)
	(-0.7,-0.5) arc (180:360:0.7 and 0.2)
	(-0.7,-0.5) -- (-0.7,0.6)
	(0.7,-0.5) -- (0.7,0.6)
	(0,-0.7) -- (0,0.4);

\draw[dashed]
	(0.7,-0.5) arc (0:180:0.7 and 0.2);

\fill
	(0,0.4) circle (0.06);
	
\draw[->]
	(0,-0.2) -- ++(0,0.2);

\node at (-0.1,0.55) {\scriptsize 1};
\node at (0.1,0.55) {\scriptsize 2};

\draw[yshift=0.5 cm,->]
	(-100:0.4) arc (-100:-170:0.4);
\draw[yshift=0.5 cm,->]
	(-80:0.4) arc (-80:-10:0.4);
		
\end{scope}

\end{tikzpicture}
\caption{Cylinder is not suitable for subdivisible tiling.}
\label{quad_degenerate1}
\end{figure}
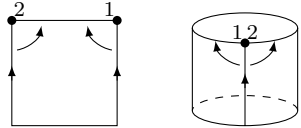

Figure \ref{qr} describes the fourth quadrilateral $R$ in the second row. In the first picture, we divide $R$ (by dashed line) into two triangles. In the second picture, we glue the two triangles in different order (single arrow edges are glued first). The result is the M\"obius band in the third picture. The fourth picture is the boundary of tile, which is a $2$-gon; we also indicate the subdivision on boundary.

\begin{figure}[h]
\centering
\begin{tikzpicture}[>=latex,scale=1]


\draw
	(-0.7,-0.7) rectangle (0.7,0.7);

\draw[dashed]
	(-0.7,0.7) -- (0.7,-0.7);
	
\draw[densely dotted]
	(-0.7,0) -- (0.7,0)
	(0,0.7) -- ++(0,0.3)
	(0,-0.7) -- ++(0,-0.3);

\foreach \a in {0,1,2}
\fill[rotate=90*\a]
	(0.7,0.7) circle (0.06);
	
\foreach \a in {1,...,4}
\node at (-45+90*\a:1.15) {\scriptsize \a};

\draw[->] (-0.1,0.7) -- ++(0.2,0);
\draw[->] (-0.7,-0.1) -- ++(0,0.2);
\draw[->] (0,0) -- ++(0.1,-0.1);
\draw[->] (-0.1,0.1) -- ++(0.1,-0.1);

\node[draw, shape=circle, inner sep=0.5] at (-0.3,-0.3) {\small $1$};
\node[draw, shape=circle, inner sep=0.5] at (0.3,0.3) {\small $2$};


\begin{scope}[xshift=3cm]

\draw
	(-1.4,0.7) -- (0,0.7) -- (0,-0.7) -- (1.4,-0.7);
	
\draw[dashed]
	(-1.4,0.7) -- (0,-0.7)
	(1.4,-0.7) -- (0,0.7);	

\draw[densely dotted]
	(0,0) -- (0.7,0)
	(-0.7,0) -- (-0.7,0.7)
	(0.7,-0.7) -- ++(0,-0.3);
	
\fill
	(0,0.7) circle (0.06)
	(0,-0.7) circle (0.06);
		
\draw[->] (0,-0.1) -- ++(0,0.2);
\draw[->] (0.7,0) -- ++(0.1,-0.1);
\draw[->] (0.6,0.1) -- ++(0.1,-0.1);
\draw[<-] (-0.7,0) -- ++(0.1,-0.1);
\draw[<-] (-0.8,0.1) -- ++(0.1,-0.1);

\node[draw, shape=circle, inner sep=0.5] at (0.4,-0.3) {\small $1$};
\node[rotate=-90, draw, shape=circle, inner sep=0.5] at (-0.4,0.3) {\small $2$};

\node at (0.25,0.8) {\scriptsize 1,2};
\node at (-0.25,-0.8) {\scriptsize 2,3};
\node at (1.5,-0.74) {\scriptsize 4};
\node at (-1.5,0.74) {\scriptsize 4};

\end{scope}


\begin{scope}[xshift=6.6cm]

\draw
	(-1.2,0.6) to[out=30,in=90] 
	(1.4,0) to[out=-90,in=-30] 
	(-1.2,-0.6)
	(-1.2,0.6) to[out=-50,in=-90] 
	(0.4,0) to[out=90,in=30] 
	(-0.5,0)
	(-1.2,0.6) to[out=-80,in=180] 
	(0.0,-0.6) to[out=-0,in=-90] 
	(0.9,0) to[out=90,in=-20] 
	(-1.2,0.6);	

\draw[dashed]
	(-0.5,0) to[out=210,in=50] (-1.2,-0.6)
	(-1.2,-0.6) -- (-1.2,0.6);	
	
\draw[densely dotted]
	(0.9,0) to[out=140,in=10] 
	(-0.4,0.3) to[out=190,in=30] 
	(-1.2,0) to[out=-70,in=180] 
	(0,-0.8) to[out=0,in=220] 
	(1.4,0)
	(0.4,0) -- ++(-0.3,0);	
			
\fill
	(-1.2,0.6) circle (0.06);

\draw[->]
	(0.9,-0.1) -- ++(0,0.1);
\draw[<-] 
	(-1.2,-0.2) -- ++(0,0.2);
\draw[<-] 
	(-1.2,0) -- ++(0,0.2);

\node at (-1.6,0.7) {\scriptsize 1,2,3};
\node at (-1.3,-0.7) {\scriptsize 4};

\end{scope}


\begin{scope}[xshift=9 cm]

\draw
	(0,-0.7) to[out=30,in=-30] (0,0.7)
	(0,-0.7) to[out=150,in=-150] (0,0.7);

\draw[densely dotted]
	(-0.36,0) -- ++(0.3,0)
	(0.36,0) -- ++(0.3,0);

\fill
	(0,0.7) circle (0.06);		
\end{scope}

\end{tikzpicture}
\caption{Degenerate quadrilateral $R$.}
\label{qr}
\end{figure}
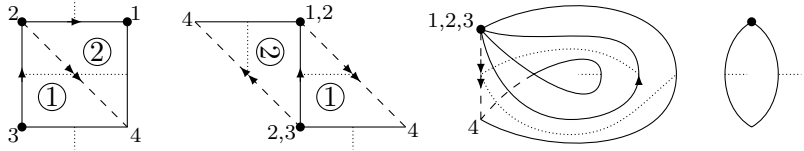

The fifth quadrilateral in the second row is a further degeneration of $R$ by identifying the vertex $4$ with the vertex $1,2,3$ (already identified in $R$). Again there are two ways $R_1,R_2$ for such a tile to embed into surfaces.

In each of the third row of Figure \ref{quad_degenerate}, we identify two pairs of edges. This yields surfaces $T^2,2P^2,P^2,S^2,P^2,2P^2$. Therefore we get specific surfaces tiled by only one quadrilateral. The first, second and third can be dismissed for the same reason as the first and second of the second row. The fourth and fifth can be dismissed because adjacent opposing pair of edges are identified. Only the last $K$ (for Klein bottle) remains, and Figure \ref{qk} explains that this tiling of the Klein bottle is subdivisible.

In the first of Figure \ref{qk}, we use dotted lines to divide the quadrilateral into two parts. We denote the halves of edges by $a_0,a_1,b_0,b_1$. Then the upper half of the quadrilateral is the pentagon in the second of Figure \ref{qk}, and the lower half of the quadrilateral is the pentagon in the fourth of Figure \ref{qk}. Both pentagons are obtained by identifying a pair of edges in a twisted way, and both are thus M\"obius bands. The boundaries of the pentagons consist of three edges $a_1,b_0,x$, as in the third of Figure \ref{qk}. Glueing the two pentagons along the common boundary gives the Klein bottle $2P^2$. There are three vertices. The vertex between $a_1,b_0$ has degree $4$. The vertices between $a_1,x$ and between $b_0,x$ have degree $3$.

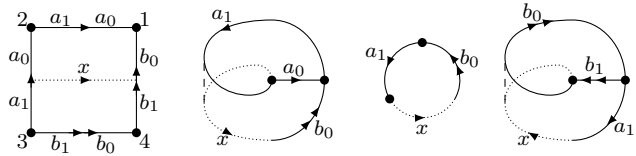
\begin{figure}[ht]
\centering
\begin{tikzpicture}[>=latex,scale=1]


\begin{scope}[xshift=-2.5 cm]

\draw
	(-0.7,-0.7) rectangle (0.7,0.7);
	
\draw[densely dotted]
	(-0.7,0) -- (0.7,0);

\foreach \a in {0,...,3}
\fill[rotate=90*\a]
	(0.7,0.7) circle (0.06);
	
\foreach \a in {1,...,4}
\node at (-45+90*\a:1.15) {\scriptsize \a};

\draw[->] (-0.1,0.7) -- ++(0.2,0);
\draw[->] (-0.7,-0.1) -- ++(0,0.2);
\draw[->] (-0.2,-0.7) -- ++(0.2,0);
\draw[->] (0,-0.7) -- ++(0.2,0);
\draw[->] (0.7,-0.2) -- ++(0,0.2);
\draw[->] (0.7,0) -- ++(0,0.2);
\draw[->] (0,0) -- ++(0.1,0);

\node at (0,0.15) {\scriptsize $x$};
\node at (-0.85,0.3) {\scriptsize $a_0$};
\node at (-0.85,-0.3) {\scriptsize $a_1$};
\node at (0.3,0.8) {\scriptsize $a_0$};
\node at (-0.3,0.8) {\scriptsize $a_1$};
\node at (0.88,0.3) {\scriptsize $b_0$};
\node at (0.88,-0.3) {\scriptsize $b_1$};
\node at (0.3,-0.85) {\scriptsize $b_0$};
\node at (-0.3,-0.85) {\scriptsize $b_1$};

\end{scope}


\foreach \a in {0,1}
{
\begin{scope}[xshift=4*\a cm]

\draw
	(0,0) -- (0.7,0)
	(0.7,0) arc (0:-90:0.7 and 0.8)
	(0.7,0) arc (0:90:0.7 and 0.8) to[out=180,in=90] 
	(-0.9,0.3) to[out=-90,in=180]
	(-0.3,-0.2) to[out=0,in=-90]
	(0,0);

\draw[densely dotted]
	(0,-0.8) to[out=180,in=-90] 
	(-0.9,-0.3) to[out=90,in=180]
	(-0.2,0.2) to[out=0,in=90]
	(0,0);	

\draw[dashed]
	(-0.9,-0.3) -- (-0.9,0.3);

\fill
	(0,0) circle (0.06)
	(0.7,0) circle (0.06);

\end{scope}
}


\draw[->]
	(-0.6,-0.69) -- ++(-30:0.1);
\node at (-0.65,-0.8) {\scriptsize $x$};

\draw[->]
	(0.3,0) -- ++(0.1,0);
\node at (0.3,0.13) {\scriptsize $a_0$};

\draw[->]
	(-0.6,0.69) -- ++(200:0.1);
\node at (-0.65,0.85) {\scriptsize $a_1$};

\draw[->]
	(0.45,-0.6) -- ++(45:0.1);
\draw[->]
	(0.58,-0.43) -- ++(55:0.1);
\node at (0.7,-0.65) {\scriptsize $b_0$};


\begin{scope}[xshift=2cm]

\draw
	(-30:0.5) arc (-30:210:0.5);

\draw[densely dotted]
	(-30:0.5) arc (-30:-150:0.5);

\fill
	(90:0.5) circle (0.06)
	(210:0.5) circle (0.06);

\draw[->]
	(155:0.5) -- (160:0.5);
\node at (150:0.7) {\scriptsize $a_1$};

\draw[->]
	(20:0.5) -- (25:0.5);
\draw[->]
	(40:0.5) -- (45:0.5);
\node at (25:0.7) {\scriptsize $b_0$};

\draw[->]
	(-90:0.5) -- (-85:0.5);
\node at (-90:0.66) {\scriptsize $x$};

\end{scope}


\begin{scope}[xshift=4cm]

\draw[<-]
	(-0.6,-0.69) -- ++(-30:0.1);
\node at (-0.65,-0.8) {\scriptsize $x$};

\draw[->]
	(0.2,0) -- ++(-0.1,0);
\draw[->]
	(0.4,0) -- ++(-0.1,0);
\node at (0.3,0.2) {\scriptsize $b_1$};

\draw[->]
	(-0.55,0.72) -- ++(20:0.1);
\draw[->]
	(-0.35,0.77) -- ++(15:0.1);
\node at (-0.53,0.9) {\scriptsize $b_0$};

\draw[<-]
	(0.45,-0.6) -- ++(45:0.1);
\node at (0.7,-0.65) {\scriptsize $a_1$};

\end{scope}

\end{tikzpicture}
\caption{Degenerate quadrilateral $K$.}
\label{qk}
\end{figure}

Only labeled quadrilaterals in Figure \ref{quad_degenerate} are suitable for subdivisible tilings. The total number is $13$.

\subsection{Orientable Degenerate Tiles}

By Proposition \ref{quad_theroem1}, only $Q,Q_{13},Q_{13,24}$ can be tiles in subdivisible tilings of orientable surfaces. The first of Figure \ref{q13} is $Q_{13}$, and the second is a disk neighborhood of the identified vertex $1,3$. The union $\hat{Q}_{13}$ of the two means that corners labeled by $\circled{1}$ and $\circled{3}$ in the quadrilateral are the two fans in the disk neighborhood. To help in understanding the picture, we imagine walking around the quadrilateral in the direction $1\to 2\to 3\to 4\to 1$. We go into the  vertex $1$ at $i_1$ ($i$ for {\em into}) and come out of $1$ at $o_1$ ($o$ for {\em out}). We use similar notations around the other identified vertices. We also note that, to satisfy Proposition \ref{quad_theroem1}, the direction $i_1\to o_1$ around the circle in the second picture must be the same as the direction $i_3\to o_3$ around the circle.

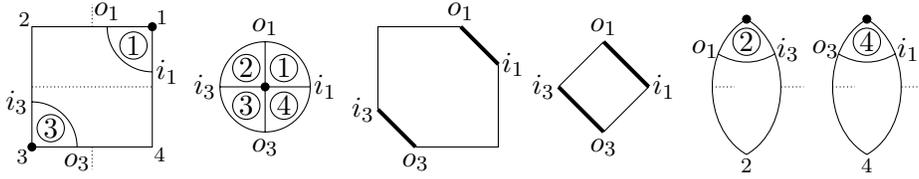
\begin{figure}[h]
\centering
\begin{tikzpicture}[>=latex,scale=1]


\draw
	(-0.8,-0.8) rectangle (0.8,0.8)
	(0.2,0.8) arc (180:270:0.6)
	(-0.2,-0.8) arc (0:90:0.6);	

\draw[densely dotted]
	(-0.8,0) -- (0.8,0)
	(0,0.8) -- ++(0,0.3)
	(0,-0.8) -- ++(0,-0.3);

\fill
	(0.8,0.8) circle (0.06)
	(-0.8,-0.8) circle (0.06);
	
\node at (1,0.2) {\small $i_1$};
\node at (0.2,1) {\small $o_1$};
\node at (-1,-0.2) {\small $i_3$};
\node at (-0.2,-1) {\small $o_3$};
		
\node at (0.93,0.93) {\scriptsize $1$};
\node at (-0.9,0.9) {\scriptsize $2$};
\node at (-0.93,-0.93) {\scriptsize $3$};
\node at (0.9,-0.9) {\scriptsize $4$};

\node[draw, shape=circle, inner sep=0.5] at (0.55,0.55) {\small $1$};
\node[draw, shape=circle, inner sep=0.5] at (-0.55,-0.55) {\small $3$};


\begin{scope}[xshift=2.3cm]

\draw
	(0,0) circle (0.6)
	(-0.6,0) -- (0.6,0)
	(0,-0.6) -- (0,0.6);	

\fill
	(0,0) circle (0.06);
	
\node at (0.8,0) {\small $i_1$};
\node at (0,0.8) {\small $o_1$};
\node at (-0.8,0) {\small $i_3$};
\node at (0,-0.8) {\small $o_3$};

\node[draw, shape=circle, inner sep=0.5] at (0.25,0.25) {\small $1$};
\node[draw, shape=circle, inner sep=0.5] at (-0.25,0.25) {\small $2$};
\node[draw, shape=circle, inner sep=0.5] at (-0.25,-0.25) {\small $3$};
\node[draw, shape=circle, inner sep=0.5] at (0.25,-0.25) {\small $4$};
			
\end{scope}


\begin{scope}[xshift=4.6cm]

\draw
	(-0.3,-0.8) -- (0.8,-0.8) -- (0.8,0.3)
	(0.3,0.8) -- (-0.8,0.8) -- (-0.8,-0.3);	

\draw[line width=1.5]
	(-0.3,-0.8) -- (-0.8,-0.3)
	(0.3,0.8) -- (0.8,0.3);

\node at (1,0.3) {\small $i_1$};
\node at (0.3,1) {\small $o_1$};
\node at (-1,-0.3) {\small $i_3$};
\node at (-0.3,-1) {\small $o_3$};

\end{scope}


\begin{scope}[xshift=6.8cm]

\draw
	(0,-0.6) -- (0.6,0)
	(0,0.6) -- (-0.6,0);	

\draw[line width=1.5]
	(-0.6,0) -- (0,-0.6)
	(0.6,0) -- (0,0.6);

\node at (0.8,0) {\small $i_1$};
\node at (0,0.8) {\small $o_1$};
\node at (-0.8,0) {\small $i_3$};
\node at (0,-0.8) {\small $o_3$};

\end{scope}


\begin{scope}[xshift=9.5cm]

\foreach \a in {-1,1}
{
\begin{scope}[xshift=0.8*\a cm]

\draw
	(0,-0.9) to[out=30,in=-30] (0,0.9)
	(0,-0.9) to[out=150,in=-150] (0,0.9)
	(-0.38,0.44) to[out=-30,in=210] (0.38,0.44);

\draw[densely dotted]
	(-0.46,0) -- ++(0.3,0)
	(0.46,0) -- ++(0.3,0);

\fill
	(0,0.9) circle (0.06);
			
\end{scope}
}

\node at (-0.8,-1.05) {\scriptsize 2};
\node at (-1.35,0.5) {\small $o_1$};
\node at (-0.25,0.5) {\small $i_3$};
\node[draw, shape=circle, inner sep=0.5] at (-0.8,0.6) {\small $2$};

\node at (0.8,-1.05) {\scriptsize 4};
\node at (1.35,0.5) {\small $i_1$};
\node at (0.25,0.5) {\small $o_3$};
\node[draw, shape=circle, inner sep=0.5] at (0.8,0.6) {\small $4$};

\end{scope}

\end{tikzpicture}
\caption{$\hat{Q}_{13}=S^2-2D^2$, and its boundary.}
\label{q13}
\end{figure}

The union of the first and second of Figure \ref{q13} can be more easily interpreted as glueing the third and fourth pictures along the indicated thick lines, so that the labels $i_*$ and $o_*$ match. This is a cylinder, or the sphere with two disks removed. The boundary of the cylinder consists of two circles, illustrated by the fifth picture. At the top of two circles, we also indicate the corners to fit into the second picture. The boundary is obtained by combining the normal edges $i_1-4-o_3$, $o_1-2-i_3$ of the third picture (``$-$'' for edges of the quadrilateral), and the normal edges $i_1\frown o_3$, $o_1\frown i_3$ of the fourth picture (``$\frown$'' for edges of the disk). The two boundary circles are $o_1-2-i_3\frown o_1$ and $i_1-4-o_3\frown i_1$.

Each boundary circle is the boundary of the subdivisible tiling $T_2$ in Figure \ref{subdiv_example}. Therefore the union $Q_{13}\cup T_2\cup T_2$ is a subdivisible tiling of the sphere by one $Q_{13}$ and twelve non-degnerate tiles $Q$. By connected sum constructions, we find that $Q_{13}$ can appear in subdivisible tilings of any surface. 

Next we turn to $Q_{13,24}$. There are two neighborhoods, one around the $\bullet$-vertex $1,3$ and another around the $\circ$-vertex $2,4$. The union $\hat{Q}_{13,24}$ of the tile with two neighborhoods is the union of the first three pictures in Figure \ref{q1324}. By combining the normal edges of the first three pictures, we get the boundary $o_1-i_2\frown o_4-i_1\frown o_3-i_4\frown o_2-i_3\frown o_1-i_2\frown o_4$ of $\hat{Q}_{13,24}$ in the fourth picture. Since the boundary is one circle, we have $\hat{Q}_{13,24}=S-D^2$ for an oriented surface $S$. By $\chi(S)=\chi(\hat{Q}_{13,24})+\chi(D^2)-\chi(S^1)=(3-4)+1-0=0$, we find $S=T^2$. Therefore $\hat{Q}_{13,24}=T^2-D^2$.

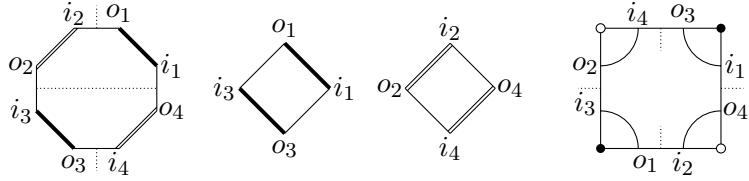
\begin{figure}[h]
\centering
\begin{tikzpicture}[>=latex,scale=1]


\foreach \a in {1,-1}
{
\begin{scope}[scale=\a]

\draw
	(0.3,0.8) -- (-0.3,0.8) -- (-0.8,0.3)
	(-0.25,0.8) -- (-0.8,0.25)
	(0.8,-0.3) -- (0.8,0.3);	

\draw[line width=1.5]
	(0.3,0.8) -- (0.8,0.3);

\end{scope}
}

\draw[densely dotted]
	(-0.8,0) -- (0.8,0)
	(0,0.8) -- ++(0,0.3)
	(0,-0.8) -- ++(0,-0.3);
	
\node at (1,0.3) {\small $i_1$};
\node at (0.3,1) {\small $o_1$};
\node at (-1,0.3) {\small $o_2$};
\node at (-0.3,1) {\small $i_2$};
\node at (-1,-0.3) {\small $i_3$};
\node at (-0.3,-1) {\small $o_3$};
\node at (1,-0.3) {\small $o_4$};
\node at (0.3,-1) {\small $i_4$};


\begin{scope}[xshift=2.5cm]

\draw
	(0,-0.6) -- (0.6,0)
	(0,0.6) -- (-0.6,0);	

\draw[line width=1.5]
	(-0.6,0) -- (0,-0.6)
	(0.6,0) -- (0,0.6);

\node at (0.8,0) {\small $i_1$};
\node at (0,0.8) {\small $o_1$};
\node at (-0.8,0) {\small $i_3$};
\node at (0,-0.8) {\small $o_3$};

\end{scope}


\begin{scope}[xshift=4.7cm]

\foreach \a in {0,...,3}
\draw[rotate=90*\a]
	(0.6,0) -- (0,0.6);	

\draw
	(0.03,0.57) -- (-0.57,-0.03)
	(-0.03,-0.57) -- (0.57,0.03);

\node at (0,0.8) {\small $i_2$};
\node at (-0.8,0) {\small $o_2$};
\node at (0,-0.8) {\small $i_4$};
\node at (0.8,0) {\small $o_4$};

\end{scope}


\begin{scope}[xshift=7.5cm]

\foreach \a in {0,...,3}
\draw[rotate=90*\a]
	(0.8,-0.8) -- (0.8,0.8)
	(0.8,0.3) arc (-90:-180:0.5);	

\draw[densely dotted]
	(-0.8,0) -- ++(-0.3,0)
	(0.8,0) -- ++(0.3,0)
	(0,0.8) -- ++(0,-0.3)
	(0,-0.8) -- ++(0,0.3);
	
\fill
	(0.8,0.8) circle (0.06)
	(-0.8,-0.8) circle (0.06);
\filldraw[fill=white]
	(-0.8,0.8) circle (0.06)
	(0.8,-0.8) circle (0.06);
	
\node at (1,0.3) {\small $i_1$};
\node at (0.3,1) {\small $o_3$};
\node at (-1,0.3) {\small $o_2$};
\node at (-0.3,1) {\small $i_4$};
\node at (-1,-0.2) {\small $i_3$};
\node at (-0.2,-1) {\small $o_1$};
\node at (1,-0.3) {\small $o_4$};
\node at (0.3,-1) {\small $i_2$};

\end{scope}

\end{tikzpicture}
\caption{$\hat{Q}_{13,24}=T^2-D^2$, and its boundary.}
\label{q1324}
\end{figure}

The boundary of $\hat{Q}_{13,24}$ can be filled by a single copy of $Q_{13,24}$. Therefore we get a subdivisible tiling $Q_{13,24}\cup Q_{13,24}$ of torus, consisting of exactly two $Q_{13,24}$ and no other tiles (and exactly two vertices, both of degree $4$). On the other hand, we can also use $T_4$ in Figure \ref{subdiv_example} to get a subdivisible tiling $Q_{13,24}\cup T_4$ of a torus. Then we may use connected sum to show that $Q_{13,24}$ can appear in subdivisible tilings of surfaces of the form $T^2\# S$, where $S$ can be any surface. 

Surfaces that are not of the form $T^2\# S$ are $S^2,P^2,2P^2$. If $Q_{13,24}$ appears in a subdivisible tiling of any such surface, then the surface must contain $\hat{Q}_{13,24}=T^2-D^2$. Since this is not the case for these three surfaces, we conclude that surfaces in which $Q_{13,24}$ can appear in subdivisible tilings are $kT^2$ ($k\ge 1$) and $kP^2$ ($k\ge 3$). 

We may specialise Proposition \ref{quad_theroem2} to the three orientable tiles and get the following. 

\begin{theorem}\label{sphere}
The tiles in a quadrilateral tiling of the sphere are $Q$ and $Q_{13}$. Moreover, the tiling is always subdivisible.
\end{theorem}

The fundamental group of the sphere is trivial, and hence there is no need to verify the fundamental cycle condition.

\begin{theorem}\label{torus}
The tiles in a quadrilateral tiling of the torus are $Q$, $Q_{13}$ and $Q_{13,24}$. Moreover, the tiling is subdivisible if and only if one of the following holds:
\begin{enumerate}
\item All tiles are $Q$ or $Q_{13}$, and the two fundamental cycles are represented by cycles with an even number of edges.
\item $Q_{13,24}$ is a tile.
\end{enumerate}
\end{theorem}

If $Q_{13,24}$ is a tile, then the two fundamental cycles are represented by the cycles $1-2-3$ and $2-3-4$. Both have an even number of edges.

\subsection{Non-orientable Degenerate Tiles}

If a subdivisible tiling contains tiles other than $Q,Q_{13},Q_{13,24}$, then the surface is not orientable. The detailed analysis about which tile can appear on what surface is more complicated, but still follows the method for the orientable case.

The first of Figure \ref{q12} is the degenerate quadrilateral $Q_{12}$, and the second is a disk neighborhood of the identified vertex $1,2$. Note that by Proposition \ref{quad_theroem1}, along the boundary of the disk neighborhood, $i_1\to o_1$ and $i_2\to o_2$ are in opposite directions. The union $\hat{Q}_{12}$ of the two is the glueing of the third and fourth pictures along the indicated thick lines. The boundary of $\hat{Q}_{12}$ is one circle $i_1\frown i_1-o_1\frown o_2-3-4-i_1$, given by the fifth picture. Therefore $\hat{Q}_{12}=S-D^2$ for a non-orientable surface $S$, with $\chi(S)=\chi(\hat{Q}_{12})+\chi(D^2)-\chi(S^1)=(2-2)+1-0=1$. This implies $S=P^2$, and $\hat{Q}_{12}=P^2-D^2$ is a M\"obius band.

\begin{figure}[h]
\centering
\begin{tikzpicture}[>=latex,scale=1]


\draw
	(-0.8,-0.8) rectangle (0.8,0.8)
	(0.2,0.8) arc (180:270:0.6)
	(-0.2,0.8) arc (0:-90:0.6);	

\draw[densely dotted]
	(-0.8,0) -- (0.8,0)
	(0,0.8) -- ++(0,0.3)
	(0,-0.8) -- ++(0,-0.3);
			
\fill
	(0.8,0.8) circle (0.06)
	(-0.8,0.8) circle (0.06);
	
\node at (1,0.2) {\small $i_1$};
\node at (0.2,1) {\small $o_1$};
\node at (-0.2,1) {\small $i_2$};
\node at (-1,0.2) {\small $o_2$};
		
\node at (0.93,0.93) {\scriptsize $1$};
\node at (-0.93,0.93) {\scriptsize $2$};
\node at (-0.9,-0.9) {\scriptsize $3$};
\node at (0.9,-0.9) {\scriptsize $4$};

\node[draw, shape=circle, inner sep=0.5] at (0.55,0.55) {\small $1$};
\node[draw, shape=circle, inner sep=0.5] at (-0.55,0.55) {\small $2$};


\begin{scope}[xshift=2.5cm]

\draw
	(0,0) circle (0.6)
	(-0.6,0) -- (0.6,0)
	(0,-0.6) -- (0,0.6);	

\fill
	(0,0) circle (0.06);
	
\node at (0.8,0) {\small $i_1$};
\node at (0,0.8) {\small $o_1$};
\node at (0,-0.8) {\small $i_2$};
\node at (-0.8,0) {\small $o_2$};

\node[draw, shape=circle, inner sep=0.5] at (0.25,0.25) {\small $1$};
\node[draw, shape=circle, inner sep=0.5] at (-0.25,-0.25) {\small $2$};
			
\end{scope}


\begin{scope}[xshift=5cm]

\draw
	(-0.8,0.3) -- (-0.8,-0.8) -- (0.8,-0.8) -- (0.8,0.3)
	(-0.3,0.8) -- (0.3,0.8);	

\draw[line width=1.5]
	(-0.3,0.8) -- (-0.8,0.3)
	(0.3,0.8) -- (0.8,0.3);

\node at (1,0.3) {\small $i_1$};
\node at (0.3,1) {\small $o_1$};
\node at (-0.3,1) {\small $i_2$};
\node at (-1,0.3) {\small $o_2$};

\end{scope}


\begin{scope}[xshift=7.3cm]

\draw
	(0,-0.6) -- (0.6,0)
	(0,0.6) -- (-0.6,0);	

\draw[line width=1.5]
	(-0.6,0) -- (0,-0.6)
	(0.6,0) -- (0,0.6);

\node at (0.8,0) {\small $i_1$};
\node at (0,0.8) {\small $o_1$};
\node at (-0.8,0) {\small $o_2$};
\node at (0,-0.8) {\small $i_2$};

\end{scope}


\begin{scope}[xshift=9.6cm]

\draw
	(-0.8,-0.8) rectangle (0.8,0.8)
	(0.2,0.8) arc (180:270:0.6)
	(-0.2,0.8) arc (0:-90:0.6);	

\draw[densely dotted]
	(-0.8,0) -- ++(-0.3,0)
	(0.8,0) -- ++(0.3,0)
	(0,0.8) -- ++(0,-0.3)
	(0,-0.8) -- ++(0,0.3);
			
\fill
	(0.8,0.8) circle (0.06)
	(-0.8,0.8) circle (0.06);
	
\node at (1,0.2) {\small $i_1$};
\node at (0.2,1) {\small $i_2$};
\node at (-0.2,1) {\small $o_1$};
\node at (-1,0.2) {\small $o_2$};
		
\node at (-0.9,-0.9) {\scriptsize $3$};
\node at (0.9,-0.9) {\scriptsize $4$};

\end{scope}

\end{tikzpicture}
\caption{$\hat{Q}_{12}=P^2-D^2$.}
\label{q12}
\end{figure}
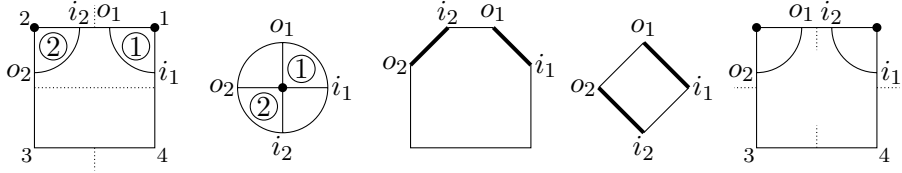

We may determine the other degenerate tiles in similar way. When three vertices $1,2,3$ are identified, the neighborhood of the identified vertex has two possible configurations, $123$ and $132$. See the second and third of Figure \ref{q123}. 
If we glue the first and second pictures, then we get two boundary circles $i_1-4-o_3\frown i_1$ and $o_1-i_2\frown i_3-o_2\frown o_1$. This implies $\hat{Q}_{123}=S-2D^2$, with $\chi(S)=\chi(Q_{123})+2=(2-3)+2=1$. Therefore $\hat{Q}_{123}=P^2-2D^2$ is the punctured M\"obius band. If we glue the first and third pictures, then we get one boundary circle $i_1-4-o_3\frown o_2-i_3\frown o_1-i_2\frown i_1$, and $\hat{Q}_{132}=2P^2-D^2$ is the punctured Klein bottle.

\begin{figure}[h]
\centering
\begin{tikzpicture}[>=latex,scale=1]

\draw
	(0.8,0.3) -- (0.8,-0.8) -- (-0.3,-0.8)
	(-0.8,-0.3) -- (-0.8,0.3)
	(-0.3,0.8) -- (0.3,0.8);	

\draw[line width=1.5]
	(-0.3,0.8) -- (-0.8,0.3)
	(0.3,0.8) -- (0.8,0.3)
	(-0.3,-0.8) -- (-0.8,-0.3);

\node at (1,0.3) {\small $i_1$};
\node at (0.3,1) {\small $o_1$};
\node at (-1,0.3) {\small $o_2$};
\node at (-0.3,1) {\small $i_2$};
\node at (-1,-0.3) {\small $i_3$};
\node at (-0.3,-1) {\small $o_3$};


\begin{scope}[xshift=2.5 cm]

\foreach \a in {0,1,2}
{
\begin{scope}[rotate=120*\a]

\draw
	(0:0.7) -- (60:0.7);

\draw[line width=1.5]
	(60:0.7) -- (120:0.7);

\end{scope}
}	

\node at (60:0.9) {\small $i_1$};
\node at (120:0.9) {\small $o_1$};
\node at (180:0.9) {\small $o_2$};
\node at (60:-0.9) {\small $i_2$};
\node at (120:-0.9) {\small $i_3$};
\node at (0:0.9) {\small $o_3$};

\node at (0,0) {123};

\end{scope}


\begin{scope}[xshift=5 cm]

\foreach \a in {0,1,2}
{
\begin{scope}[rotate=120*\a]

\draw
	(0:0.7) -- (60:0.7);

\draw[line width=1.5]
	(60:0.7) -- (120:0.7);

\end{scope}
}	

\node at (60:0.9) {\small $i_1$};
\node at (120:0.9) {\small $o_1$};
\node at (180:0.9) {\small $i_3$};
\node at (60:-0.9) {\small $o_3$};
\node at (120:-0.9) {\small $o_2$};
\node at (0:0.9) {\small $i_2$};

\node at (0,0) {132};
\
\end{scope}
	
\end{tikzpicture}
\caption{$\hat{Q}_{123}=P^2-2D^2$ and $\hat{Q}_{132}=2P^2-D^2$.}
\label{q123}
\end{figure}
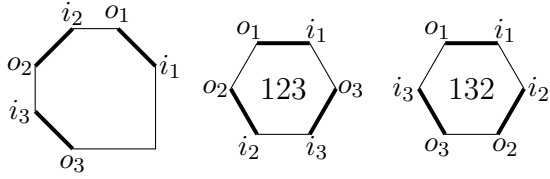

When all four vertices $1,2,3,4$ are identified, the neighborhood configuration is determined by the shape of the cycle connecting its four vertices. The cycle $1234$ has the shape $\square$, and $1243$ has the shape $\Join$. This gives two degenerate tiles, $Q_{1234}$ and $Q_{1243}$. For further degeneracy of $R$, the identification of $4$ and $1,2,3$ can be normal or twisted, which gives $R_1$ and $R_2$, respectively.

\begin{table}[ht]
\centering
\begin{tabular}{|c|c|c|c|}
\hline 
Tile & Boundary & Union with Nbhd & Min Surface  \\
\hline 
$Q_{12}$ & $4_{12}$ & $P^2-D^2$ & $P^2$ \\
\hline 
$Q_{13}$ & $2_1+2_1$ & $S^2-2D^2$ & $S^2$ \\
\hline 
$Q_{123}$ & $2_1+2_{12}$ & $P^2-2D^2$ & $P^2$ \\
\hline 
$Q_{132}$ & $4_{123}$ & $2P^2-D^2$ & $2P^2$ \\
\hline 
$Q_{1234}$ & $2_{12}+2_{12}$ & $2P^2-2D^2$ & $2P^2$ \\
\hline 
$Q_{1243}$ & $4_{1234}$ & $3P^2-D^2$ & $3P^2$ \\
\hline 
$Q_{12,34}$ & $4_{12,34}$ & $2P^2-D^2$ & $2P^2$ \\
\hline 
$Q_{13,24}$ & $4_{13,24}$ & $T^2-D^2$ & $T^2$ \\
\hline 
$R$ & $2_1$ & $P^2-D^2$  & $P^2$ \\
\hline 
$R_1$ & $1_1+1_1$ & $P^2-2D^2$ & $2P^2$ \\
\hline 
$R_2$ & $2_{12}$ & $2P^2-D^2$  & $2P^2$ \\
\hline 
$K$ & $\emptyset$ & $2P^2$  & \\
\hline 
\end{tabular}
\caption{Degenerate tiles and the minimal surfaces they embed into.}
\label{all}
\end{table}

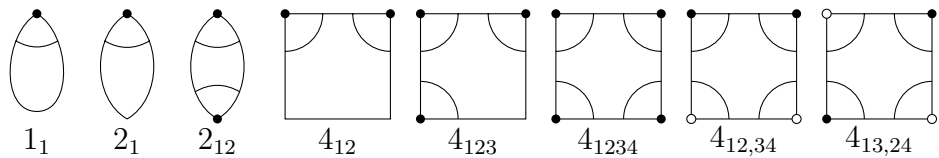
\begin{figure}[h]
\centering
\begin{tikzpicture}[>=latex,scale=1]


\begin{scope}[xshift=-1.8cm]

\draw
	(0,0.7) to[out=210,in=180] 
	(0,-0.6) to[out=0,in=-30] 
	(0,0.7)
	(-0.28,0.34) to[out=-30,in=210] (0.28,0.34);	
			
\fill
	(0,0.7) circle (0.06);

\node at (0,-1) {$1_1$};

\end{scope}


\foreach \a in {-1,1}
{
\begin{scope}[xshift=0.6*\a cm]

\draw
	(0,-0.7) to[out=30,in=-30] (0,0.7)
	(0,-0.7) to[out=150,in=-150] (0,0.7)
	(-0.3,0.34) to[out=-30,in=210] (0.3,0.34);

\fill
	(0,0.7) circle (0.06);
			
\end{scope}
}

\draw[xshift=0.6cm]
	(-0.3,-0.34) to[out=30,in=150] (0.3,-0.34);

\fill
	(0.6,-0.7) circle (0.06);

\node at (-0.6,-1) {$2_1$};
\node at (0.6,-1) {$2_{12}$};


\foreach \a in {0,1,2,3,4} 
\draw[xshift=2.2 cm + 1.8*\a cm]
	(-0.7,-0.7) rectangle (0.7,0.7)
	(0.2,0.7) arc (180:270:0.5)
	(-0.2,0.7) arc (0:-90:0.5);

\foreach \a in {1,2,3,4} 
\draw[xshift=2.2 cm + 1.8*\a cm]
	(-0.2,-0.7) arc (0:90:0.5);

\foreach \a in {2,3,4} 
\draw[xshift=2.2 cm + 1.8*\a cm]
	(0.2,-0.7) arc (180:90:0.5);
	

\foreach \a in {0,1,2,3,4} 
\fill[xshift=2.2 cm + 1.8*\a cm]
	(0.7,0.7) circle (0.06);

\foreach \a in {0,1,2,3} 
\fill[xshift=2.2 cm + 1.8*\a cm]
	(-0.7,0.7) circle (0.06);	
\filldraw[xshift=2.2 cm + 1.8*4 cm, fill=white]
	(-0.7,0.7) circle (0.06);

\foreach \a in {1,2,4} 
\fill[xshift=2.2 cm + 1.8*\a cm]
	(-0.7,-0.7) circle (0.06);
\filldraw[xshift=2.2 cm + 1.8*3 cm, fill=white]
	(-0.7,-0.7) circle (0.06);

\fill[xshift=2.2 cm + 1.8*2 cm]  
	(0.7,-0.7) circle (0.06);
\foreach \a in {3,4}
\filldraw[xshift=2.2 cm + 1.8*\a cm, fill=white]
	(0.7,-0.7) circle (0.06);


\node at (2.2,-1) {$4_{12}$};
\node at (4,-1) {$4_{123}$};
\node at (5.8,-1) {$4_{1234}$};
\node at (7.6,-1) {$4_{12,34}$};
\node at (9.4,-1) {$4_{13,24}$};
			
\end{tikzpicture}
\caption{Boundaries of degenerate tiles.}
\label{boundary}
\end{figure}

Table \ref{all} gives the complete list of degenerate tiles, their boundaries and the minimal surfaces they embed into. The notations for boundaries are given in Figure \ref{boundary}. The vertices decorated with $\bullet$ or $\circ$ are the identified vertices of the original quadrilateral, and already have degree $\ge 3$. The undecorated vertices are the original unidentified vertices, and have degree $2$. We omit the subdivisions on the boundary edges.

The fourth column in Table \ref{all} is the minimal surfaces on which the tiles may appear in some subdivisible tiling. Notice the boundaries $2_*$ and $4_*$ can be filled in by subdivisible tilings $T_2$ and $T_4$ in Figure \ref{subdiv_example}. For tiles that are not $R_1$ and $K$, therefore, we find the surface $S_{\min}$ in the fourth column of Table \ref{all}, such that the tile can appear in a subdivisible tiling of $S_{\min}$. By connected sum constructions, the surface can also appear in a subdivisible tiling of $S_{\min}\# S$ for any other surface $S$. For example, $Q_{1234}$ can be a tile in subdivisible tilings of all non-orientable surfaces except $P^2$.

We already know $K$ must be the single tile of a subdivisible tiling of $2P^2$. Therefore $K$ cannot be a tile on any other surface.

It remains to consider $R_1$. We can certainly glue two copies of $R_1$ along their common boundary to find that $2P^2$ has a subdivisible tiling consisting of exactly two $R_1$. In order to do the connected sum operation, so that $R_1$ appears in subdivisible tilings of $2P^2\# S$, we may insert a copy of $T_2$ between two $1_1$. See Figure \ref{qr1}.

\begin{figure}[h]
\centering
\begin{tikzpicture}[>=latex,scale=1]

\draw
	(0,0) circle (1)
	(0,1) to[out=210,in=180] 
	(0,0.2) to[out=0,in=-30] 
	(0,1)
	(-0.4,-0.6) -- (-0.3,-0.3) -- (-0.3,0) 
	-- (0.3,0) -- (0.3,-0.3) -- (0.4,-0.6) -- cycle
	(-0.3,-0.3) -- (0.3,-0.3)
	(0,1) to[out=-20,in=20] (0.3,0)
	(0,1) to[out=-10,in=20] (0.4,-0.6)
	(0,1) to[out=200,in=160] (-0.3,0)
	(0,1) to[out=190,in=160] (-0.4,-0.6);

\fill
	(0,1) circle (0.06);
	
\node at (0,0.6) {$1_1$};
\node at (0.9,-0.9) {$1_1$};

\end{tikzpicture}
\caption{$T_2$ between two $1_1$.}
\label{qr1}
\end{figure}
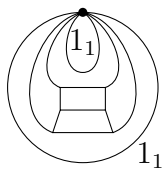

On the other hand, if $R_1$ appears in a subdivisible tiling on $P^2$, then each component of the boundary of $R_1$ must be filled by a disk. In other words, there is a subdivisible tiling of the disk with $1_1$ as the boundary. On the other hand, we note that the boundary of any tile in Table \ref{all} consists of either $2$ or $4$ edges. Therefore the boundary of any tiling with these tiles must consist of even number of edges. This implies that a single $1_1$ cannot be a boundary. We conclude that the surfaces on which $R_1$ can appear are $kP^2$, for $k\ge 2$.

By Table \ref{all}, only $Q,Q_{12},Q_{13},Q_{123},R$ can appear in a subdivisible tiling of $P^2$. We may specialise Proposition \ref{quad_theroem2} to the projective space, which has only one fundamental cycle. 

\begin{theorem}\label{torus}
The tiles in a quadrilateral tiling of projective space are $Q$, $Q_{13}$, $Q_{12}$, $Q_{123}$ and $R$. Moreover, the tiling is subdivisible if and only if one of the following holds:
\begin{enumerate}
\item All tiles are $Q$ or $Q_{13}$, and the fundamental cycle is represented by a cycle with an odd number of edges.
\item One of $Q_{12},Q_{123},R$ is a tile.
\end{enumerate}
\end{theorem}

If one of $Q_{12},Q_{123},R$ is a tile, then the loop at the identified vertex represents the fundamental cycle and is a cycle with only one edge.

\section{Pentagonal Subdivisions}
\label{pent}

By subdividing labeled quadrilaterals in Figure \ref{quad_degenerate} into two halves in either way, we find three types of pentagonal tiles in simple pentagonal subdivisions. See Figure \ref{qconverse1}.

\begin{figure}[h]
\centering
\begin{tikzpicture}[>=latex,scale=1]


\foreach \a in {-1,0,1,2}
\draw[rotate=72*\a]
	(18:0.8) -- (90:0.8);

\draw[densely dotted]
	(234:0.8) -- (-54:0.8)
	(90:0.8) -- ++(0,0.3);

\fill
	(18:0.8) circle (0.1)
	(162:0.8) circle (0.1);

\node at (90:0.6) {\small 3};
\node at (-54:0.95) {\small 3};
\node at (234:0.95) {\small 3};

\node at (0.9,0.6) {$P_1$};
		

\begin{scope}[xshift=3cm]

\draw
	(0,0.3) circle (0.5)
	(-30:0.8) arc (-30:210:0.8);	

\draw[densely dotted]
	(-30:0.8) arc (-30:-150:0.8)
	(0,-0.2) -- ++(0,0.3);

\fill
	(90:0.8) circle (0.1);

\node at (-30:0.95) {\small 3};
\node at (210:0.95) {\small 3};
\node at (0,-0.4) {\small 3};

\node at (0.9,0.6) {$P_2$};
		
\end{scope}


\begin{scope}[xshift=6cm]

\draw
	(0,0) -- (0.7,0)
	(0.7,0) arc (0:-90:0.7 and 0.8)
	(0.7,0) arc (0:90:0.7 and 0.8) to[out=180,in=90] 
	(-0.9,0.3) to[out=-90,in=180]
	(-0.3,-0.2) to[out=0,in=-90]
	(0,0);

\draw[densely dotted]
	(0,-0.8) to[out=180,in=-90] 
	(-0.9,-0.3) to[out=90,in=180]
	(-0.2,0.2) to[out=0,in=90]
	(0,0);	

\draw[dashed]
	(-0.9,-0.3) -- (-0.9,0.3);

\fill
	(0.7,0) circle (0.1);

\node at (0.15,0.2) {\small 3};
\node at (0,-0.6) {\small 3};

\node at (0.9,0.6) {$P_3$};
		
\end{scope}

\end{tikzpicture}
\caption{Pentagonal tiles in simple pentagonal subdivision.}
\label{qconverse1}
\end{figure}
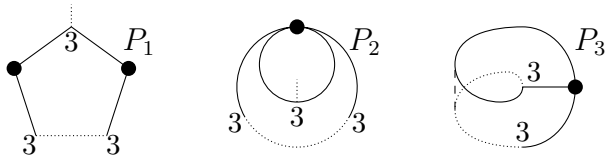

The pentagon $P_1$ is non-degenerate, and $P_2,P_3$ are degenerate. The big $\bullet$ (unlike the small $\bullet$ used in Section \ref{quad}) indicates vertices in the original quadrilateral tiling. If we include the neighborhood disk at the $\bullet$-vertex, then $P_2$ is a M\"obius band. Of course, $P_3$ is already a M\"obius band. The number ``3'' indicates new vertices at the middle of edges of quadrilaterals. These vertices must have degree $3$. The following lists how the pentagons are obtained by subdividing a quadrilateral.
\begin{enumerate}
\item $P_1$: $Q,Q_{12},Q_{13},Q_{123},Q_{132},Q_{12,34},Q_{13,24},R$.
\item $P_2$: $Q_{12},Q_{123},Q_{132},Q_{1234},Q_{1243},Q_{12,34},R_1,R_2$.
\item $P_3$: $R,R_1,R_2,K$.
\end{enumerate}

We answer the question when a pentagonal tiling is the result of simple pentagonal subdivision. Then we apply the conclusion to the pentagonal subdivision and double pentagonal subdivision introduced in \cite{wy1}.

\subsection{Criterion for Simple Pentagonal Subdivision}

The only pentagon in a simple pentagonal subdivision on an orientable surface is $P_1$. The converse is also true.

\begin{theorem}\label{sp_thm1}
A pentagonal tiling of an orientable surface is the simple pentagonal subdivision of a quadrilateral tiling if and only if it is possible to label some vertices by $\bullet$, such that the unlabelled vertices have degree $3$, and every tile is $P_1$ in Figure \ref{qconverse1}.
\end{theorem}

\begin{proof}
There is a unique edge in $P_1$ such that both ends are unlabelled. We indicate this edge by a dotted line. We start with the tile labeled $\circled{1}$ in Figure \ref{qconv_orient}. The dotted edge is shared with another tile $\circled{2}$. By the conditions of the theorem, the $\bullet$-vertices of $\circled{2}$ must be as indicated in Figure \ref{qconv_orient}. The union of $\circled{1}$ and $\circled{2}$ is a quadrilateral.

\begin{figure}[h]
\centering
\begin{tikzpicture}[>=latex,scale=1]

\foreach \a in {-1,1}
{
\begin{scope}[yscale=\a]

\draw
	(-0.6,0) -- (-0.8,0.8) -- (0,1) -- (0.8,0.8) -- (0.6,0)
	(0.8,0.8) -- (1.6,0.6) -- (2.4,0.8) -- (2.6,0);

\fill
	(-0.8,0.8) circle (0.1)
	(0.8,0.8) circle (0.1)
	(2.4,0.8) circle (0.1);
	
\end{scope}
}

\draw[densely dotted]
	(-0.6,0) -- (0.6,0)
	(1.6,0.6) -- (1.6,-0.6);

\node[draw, shape=circle, inner sep=0.5] at (0,0.5) {\small $1$};
\node[draw, shape=circle, inner sep=0.5] at (0,-0.5) {\small $2$};
\node[draw, shape=circle, inner sep=0.5] at (1.1,0) {\small $3$};
\node[draw, shape=circle, inner sep=0.5] at (2.1,0) {\small $4$};

\end{tikzpicture}
\caption{When does a pentagonal tiling come from simple pentagonal subdivision?}
\label{qconv_orient}
\end{figure}
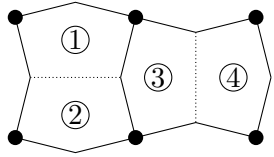

The two ends of the dotted edge are degree $3$ vertices shared by $\circled{1}$ and $\circled{2}$. Outside one such vertex is a tile $\circled{3}$. The $\bullet$-vertices of $\circled{3}$ are already given as vertices of $\circled{1}$ and $\circled{2}$. Using $\circled{3}$ in place of $\circled{1}$, we get another tile $\circled{4}$ similar to $\circled{2}$, and the union of $\circled{3}$ and $\circled{4}$ is another quadrilateral.

For each dotted edge, we get a quadrilateral by taking the union of two tiles on two sides of the edge. These quadrilaterals may be degenerate, but cannot have adjacent opposing edges identified, as indicated in Figure \ref{degenerate}. Therefore the quadrilaterals form a tiling. Moreover, Figure \ref{qconv_orient} shows that the pentagonal tiling is the simple pentagonal subdivision of this quadrilateral tiling.
\end{proof}

We note that the proof (of sufficiency) actually does not require the surface to be orientable. Therefore the sufficiency part of Theorem \ref{sp_thm1} is also valid on non-orientable surfaces. The following is the more general result taking into account also non-orientable tiles.

\begin{theorem}\label{sp_thm2}
A pentagonal tiling of a surface is the simple pentagonal subdivision of a quadrilateral tiling if and only if it is possible to label some vertices by $\bullet$, such that the unlabelled vertices have degree $3$, and the tiles are $P_1,P_2,P_3$ in Figure \ref{qconverse1}.
\end{theorem}

In this theorem, we require $P_2$ to embed into the surface as a M\"obius band. 

\begin{proof}
In each of $P_1,P_2,P_3$, there is only one edge with both ends unlabelled. As in the proof of Theorem \ref{sp_thm1}, we indicate the edge by a dotted line. Every pentagon has one dotted line, and every dotted line is shared by two pentagons. By taking the union of two pentagons on two sides of a dotted line, we get a quadrilateral. For example, Figure \ref{qconverse2} shows that $P_1\cup P_2=Q_{12}$ and $P_2\cup P_2=Q_{12,34}$ (note smaller $\bullet$ and $\circ$ used for identified vertices in Section \ref{quad}). By further identifying vertices (for example, one $\bullet$-vertex of $P_1$ identified with the $\bullet$-vertex of $P_2$), $P_1\cup P_2=Q_{12}$ may be further reduced to $Q_{123}$ or $Q_{132}$. Similar reductions happen to other unions. Then the same argument as the proof of Theorem \ref{sp_thm1} shows that the quadrilaterals form a tiling of the surface (in particular, the forbidden identification in Figure \ref{degenerate} does not happen), and the pentagonal tiling is the simple pentagonal subdivision of the quadrilateral tiling.
\end{proof}

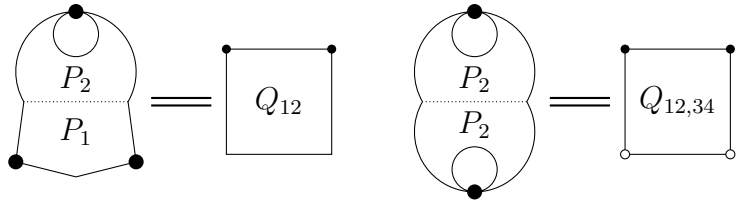
\begin{figure}[h]
\centering
\begin{tikzpicture}[>=latex,scale=1]


\begin{scope}[yshift=0.4cm]

\draw
	(0,0.5) circle (0.3)
	(-30:0.8) arc (-30:210:0.8)
	(-30:0.8) -- (0.8,-1.2) -- (0,-1.4) -- (-0.8,-1.2) -- (210:0.8);

\draw[densely dotted]
	(-30:0.8) -- (210:0.8);

\fill
	(90:0.8) circle (0.1)
	(0.8,-1.2) circle (0.1)
	(-0.8,-1.2) circle (0.1);

\node at (0,-0.1) {$P_2$};
\node at (0,-0.8) {$P_1$};		

\draw[thick]
	(1,-0.35) -- ++(0.8,0)
	(1,-0.45) -- ++(0.8,0);

\end{scope}

\begin{scope}[xshift=2.7cm]

\draw
	(-0.7,-0.7) rectangle (0.7,0.7);

\fill
	(0.7,0.7) circle (0.06)
	(-0.7,0.7) circle (0.06);
	
\node at (0,0) {$Q_{12}$};
	
\end{scope}


\foreach \a in {-1,1}
\draw[shift={(5.3cm, 0.4*\a cm)}, yscale=\a]
	(0,0.5) circle (0.3)
	(-30:0.8) arc (-30:210:0.8);

\begin{scope}[shift={(5.3cm,0.4cm)}]

\draw[densely dotted]
	(-30:0.8) -- (210:0.8);

\fill
	(90:0.8) circle (0.1)
	(-90:1.6) circle (0.1);

\node at (0,-0.1) {$P_2$};
\node at (0,-0.7) {$P_2$};		

\draw[thick]
	(1,-0.35) -- ++(0.8,0)
	(1,-0.45) -- ++(0.8,0);

\end{scope}

\begin{scope}[xshift=8cm]

\draw
	(-0.7,-0.7) rectangle (0.7,0.7);

\fill
	(0.7,0.7) circle (0.06)
	(-0.7,0.7) circle (0.06);

\filldraw[fill=white]
	(0.7,-0.7) circle (0.06)
	(-0.7,-0.7) circle (0.06);
		
\node at (0,0) {$Q_{12,34}$};
	
\end{scope}

\end{tikzpicture}
\caption{Union of two pentagons is one quadrilateral.}
\label{qconverse2}
\end{figure}

\subsection{Criterion for Pentagonal Subdivision}

Pentagonal subdivision is introduced for tilings of oriented surfaces in \cite{wy1}. Specifically, let $T$ be a tiling of an oriented surface, with possibly degenerate tiles. We divide each edge $e$ in $T$ into three segments. If $e$ is one edge of a tile $t$, we may use the orientation to label the two dividing points on $e$ as the first and the second, from the viewpoint of $t$. Note that $e$ is shared by $t$ and another tile $t'$. The labelling of $e$ from the viewpoint of $t'$ is different from the viewpoint of $t$. Then we fix a center point of $t$, and connect it to the first dividing points of the edges of $t$. We obtain the {\em pentagonal subdivision} $T(5)$ by doing this for all the pairs $(t,e)$ in the tiling. See Figure \ref{5div}.

\begin{figure}[htp]
\centering
\begin{tikzpicture}[>=latex]

\foreach \a in {0,1}
\draw[xshift=4.8*\a cm]
	(-1.5,-1.5) rectangle (0,0)
	(0,0) -- (1.5,-0.6) -- (0,-1.5)
	(0,0) -- (-0.6,1.5) -- (-1.5,0)
	(1.5,-0.6) -- (1.5,0.6) -- (0.6,1.5) -- (-0.6,1.5)
	;

	
\draw[->]
	(-0.95,-0.75) arc (-180:90:0.2);
\draw[->]
	(-0.9,0.5) arc (-180:90:0.2);
\draw[->]
	(0.3,-0.7) arc (-180:90:0.2);
\draw[->]
	(0.4,0.6) arc (-180:90:0.2);
	
\fill
	(0.5,-0.2) circle (0.04)
	(-0.2,0.5) circle (0.04)
	(1,-0.4) circle (0.04)
	(-0.4,1) circle (0.04)
	(1.5,0.2) circle (0.04)
	(0.2,1.5) circle (0.04)
	(1.5,-0.2) circle (0.04)
	(-0.2,1.5) circle (0.04)
	(0.9,1.2) circle (0.04)
	(1.2,0.9) circle (0.04);

\node at (0.5,0.8) {$t$};
\node at (-0.8,0.7) {$t'$};

\node at (0.55,-0.05) {\scriptsize 1};
\node at (0.45,-0.35) {\scriptsize 2};
\node at (1.05,-0.25) {\scriptsize 2};
\node at (0.95,-0.55) {\scriptsize 1};

\node at (-0.05,0.55) {\scriptsize 2};
\node at (-0.35,0.45) {\scriptsize 1};
\node at (-0.25,1.05) {\scriptsize 1};
\node at (-0.55,0.95) {\scriptsize 2};

\node at (1.35,0.2) {\scriptsize 2};
\node at (1.35,-0.2) {\scriptsize 1};

\node at (0.2,1.35) {\scriptsize 1};
\node at (-0.2,1.35) {\scriptsize 2};

\node at (0.8,1.1) {\scriptsize 2};
\node at (1.1,0.8) {\scriptsize 1};
			
\draw[very thick, ->]
	(2,0) -- ++(0.8,0);

\node at (1.2,-1.3) {$T$};


\begin{scope}[xshift=4.8cm]

\draw
	(0,-1) -- (-1.5,-0.5)
	(-0.5,0) -- (-1,-1.5)
	(0.5,-0.7) -- (0,-0.5) 
	(0.5,-0.7) -- (1,-0.4)
	(0.5,-0.7) -- (0.5,-1.2)
	(-0.7,0.5) -- (-1,0) 
	(-0.7,0.5) -- (-0.2,0.5)
	(-0.7,0.5) -- (-0.9,1)
	(0.55,0.55) -- (0.5,-0.2)
	(0.55,0.55) -- (-0.4,1)
	(0.55,0.55) -- (1.5,-0.2)
	(0.55,0.55) -- (0.2,1.5)
	(0.55,0.55) -- (1.2,0.9)
	;

\fill
	(0,0) circle (0.1)
	(0,-1.5) circle (0.1)
	(-1.5,0) circle (0.1)
	(-1.5,-1.5) circle (0.1)
	(-0.6,1.5) circle (0.1)
	(0.6,1.5) circle (0.1)
	(1.5,0.6) circle (0.1)
	(1.5,-0.6) circle (0.1);	

\filldraw[fill=white]
	(-0.7,0.5) circle (0.1)
	(-0.75,-0.75) circle (0.1)
	(0.5,-0.7) circle (0.1)
	(0.55,0.55) circle (0.1);

\node at (1.2,-1.3) {$T(5)$};
	
\end{scope}


\begin{scope}[xshift=8.5cm]

\foreach \a in {0,...,4}
\draw[rotate=72*\a]
	(18:0.8) -- (90:0.8);

\fill
	(18:0.8) circle (0.1);
\filldraw[fill=white]
	(162:0.8) circle (0.1);
	
\draw
	(0,0) circle (0.2);
	
\draw[<-]
	(-0.1,0.195) -- (0.05,0.195);

\node at (0,-1.2) {tile in $T(5)$};

\node at (90:0.95) {\small 3};
\node at (-54:0.95) {\small 3};
\node at (234:0.95) {\small 3};
	
\end{scope}

\end{tikzpicture}
\caption{Pentagonal subdivision.}
\label{5div}
\end{figure}
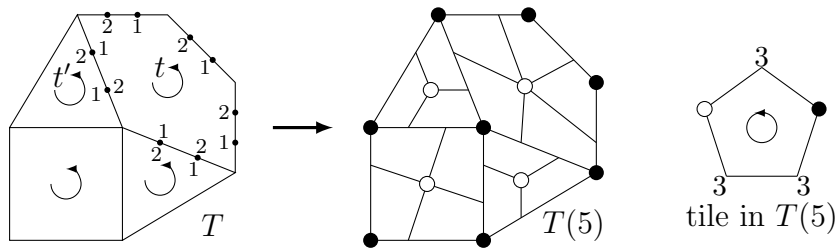

Every tile in a pentagonal subdivision $T(5)$ is non-degenerate and the vertices have the pattern $\bullet-3-\circ-3-3$ along the orientation direction. Here $\bullet$ means a vertex of $T$, and $\circ$ means the center of a tile in $T$, and $3$ means a vertex of degree $3$ (the dividing points of edges in $T$). This property actually characterises pentagonal subdivision. 

\begin{theorem}\label{ps_thm1}
A pentagonal tiling of an oriented surface is the pentagonal subdivision of a tiling, if and only if all tiles are non-degenerate, and it is possible to label some vertices by $\bullet$ and some other vertices by $\circ$, such that vertices of each tile have the pattern $\bullet-3-\circ-3-3$ along the orientation of the surface.  
\end{theorem}

We remark that a tiling $T$ and its dual $T^*$ (using the opposite of the orientation for $T$) induce the same pentagonal subdivision, with $\circ$- and $\bullet$-vertices exchanged. 

\begin{proof}
If we do not distinguish $\bullet$ and $\circ$, then we may apply Theorem \ref{sp_thm1} to find that the pentagonal tiling is a simple pentagonal subdivision of a quadrilateral tiling. Then we consider the quadrilateral tiles around a $\circ$-vertex, as in Figure \ref{proof1}. By the assumption on pentagonal tiles, we have a simple pentagonal subdivision of the neighborhood compatible with the orientation at the $\circ$-vertex. Then we combine the pentagons around the $\circ$-vertex to form a polygon (enclosed by the dashed lines). The vertices of this polygon are $\bullet$-vertices, and the edges of this polygon are $\bullet-1\cdots 2-\bullet$. We have one polygon for each $\circ$-vertex, and the polygons form a tiling of the surface. The pentagonal tiling is the pentagonal subdivision of this tiling. 
\end{proof}

\begin{figure}[h]
\centering
\begin{tikzpicture}[>=latex,scale=1]


\foreach \a in {0,...,3}
{
\begin{scope}[rotate=90*\a]

\draw
	(0,0) rectangle (1,1);

\draw[densely dotted]
	(0.5,0) -- (0.5,1);

\draw[dashed]
	(1.07,0.07) -- (0.57,0.07) -- (0.57,1.07) -- (-0.07,1.07);

\fill
	(1,0) circle (0.1);

\filldraw[fill=white]
	(1,1) circle (0.1);

\node at (0.5,-0.15) {\scriptsize 1};
\node at (0.6,1.2) {\scriptsize 2};
		
\end{scope}
}
	
\filldraw[fill=white]
	(0,0) circle (0.1);

\end{tikzpicture}
\caption{Construction of pentagonal subdivision.}
\label{proof1}
\end{figure}
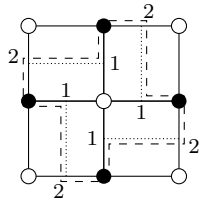

If we start with a tiling in which all vertices have degree $3$, then each tile in the pentagonal subdivision has one $\circ$-vertex, and its other four vertices (including $\bullet$-vertex) have degree $3$. Dually, if we start with a triangular tiling, then each tile in the pentagonal subdivision has one $\bullet$-vertex, and its other four vertices have degree $3$. The converse is also true.

\begin{theorem}\label{ps_thm2}
Suppose all tiles in a pentagonal tiling of a surface are non-degenerate, and it is possible to label some vertices by $\circ$, such that each tile has exactly one $\circ$-vertex, and the other four vertices have degree $3$. Then the surface is orientable, and the tiling is obtained by pentagonal subdivision.
\end{theorem}

\begin{proof}
Figure \ref{proof3} shows three consecutive tiles $\circled{1},\circled{2},\circled{3}$ around a $\circ$-vertex. Since all the vertices of $\circled{1},\circled{2},\circled{3}$ other than the central $\circ$-vertex have degree $3$, there is a tile $\circled{4}$ outside $\circled{1},\circled{2}$, a tile $\circled{5}$ outside  $\circled{2},\circled{3}$, and a tile $\circled{6}$ bordering $\circled{2},\circled{4},\circled{5}$.

We already know that the three vertices of $\circled{4}$ shared with $\circled{1},\circled{2}$ are not $\circ$-vertices. By the assumption, among the two remaining vertices of $\circled{4}$, one is $\circ$ and the other has degree $3$. Without loss of generality, we assume the $\circ$-vertex is as indicated. Then there is unique way to assign the $\bullet$-vertex shared by $\circled{1},\circled{4}$, such that $\bullet$ and $\circ$ are not adjacent in $\circled{1}$ and $\circled{4}$. We note that $\circled{1},\circled{4}$ have the same orientation following the $\bullet-3-\circ-3-3$ direction.

\begin{figure}[h]
\centering
\begin{tikzpicture}[>=latex,scale=1]
	
\foreach \a in {-1,0,1}
\draw[rotate=90*\a]
	(0,0) -- (0.6,0) -- (1,0.5) -- (0.5,1) -- (0,0.6) -- (0,0);

\foreach \a in {0,1}
\draw[rotate=90*\a]
	(1,0.5) -- (1.5,0.5)
	(1,-0.5) -- (1.5,-0.5);
	
\draw
	(1.5,-0.5) -- (1.5,1.5) -- (-0.5,1.5);

\fill
	(0.5,1) circle (0.1)
	(1,-0.5) circle (0.1);

\filldraw[fill=white]
	(0,0) circle (0.1)
	(1.5,0.5) circle (0.1)
	(-0.5,1.5) circle (0.1);

\node[draw, shape=circle, inner sep=0.5] at (-45:0.6) {\small $1$};
\node[draw, shape=circle, inner sep=0.5] at (45:0.6) {\small $2$};
\node[draw, shape=circle, inner sep=0.5] at (135:0.6) {\small $3$};
\node[draw, shape=circle, inner sep=0.5] at (1.1,0) {\small $4$};
\node[draw, shape=circle, inner sep=0.5] at (0,1.1) {\small $5$};
\node[draw, shape=circle, inner sep=0.5] at (1.1,1.1) {\small $6$};

\end{tikzpicture}
\caption{Every pentagon has one $\circ$ and four degree $3$ vertices.}
\label{proof3}
\end{figure}
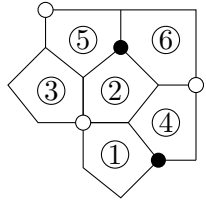

The $\circ$-vertex of $\circled{4}$ is also the unique $\circ$-vertex of $\circled{6}$. Therefore the other four vertices of $\circled{6}$ have degree $3$. Then we already know four non $\circ$-vertices for $\circled{5}$. This determines the $\circ$-vertex of $\circled{5}$ as indicated. We can assign the $\bullet$-vertex shared by $\circled{2},\circled{5},\circled{6}$, as before. Then $\circled{1},\circled{2},\circled{4},\circled{5},\circled{6}$ have the same orientation following the $\bullet-3-\circ-3-3$ direction. 

We started from the central (first) $\circ$-vertex, made assumption (by symmetry) on the second $\circ$-vertex, and then derived the third $\circ$-vertex. Since the position of the third $\circ$-vertex with respect to the first $\circ$-vertex is the same as the position of the second $\circ$-vertex with respect to the first $\circ$-vertex, and is also the same as the position of the first $\circ$-vertex with respect to the second $\circ$-vertex, the argument can continue in two directions and eventually covers the whole tiling. In particular, we introduce compatible  orientations of tiles following the $\bullet-3-\circ-3-3$ direction. This makes the surface orientable. Moreover, we also assign all the $\bullet$-vertices. Then we may apply Theorem \ref{ps_thm1} to show that the pentagonal tiling is a pentagonal subdivision.
\end{proof}

\subsection{Criterion for Double Pentagonal Subdivision}
\label{double}

Double pentagonal subdivision is also introduced for any tiling of oriented surface in \cite{wy1}. This is actually a combination of two subdivisions. First we apply {\em quadrilateral subdivision} given by Figure \ref{4div}. Then we apply simple pentagonal subdivision to the quadrilateral subdivision.

The quadrilateral subdivision $T(4)$ takes any tiling $T$, select a middle point of each edge, select a center point of each tile, and then connect the center point of any tile to the middle point of each boundary edge of the tile. It is the union of $T$ with its dual $T^*$. Unlike pentagonal subdivision, we do not need orientation to construct quadrilateral subdivision.

\begin{figure}[htp]
\centering
\begin{tikzpicture}[>=latex]

\foreach \a in {0,1}
\draw[xshift=4.4*\a cm]
	(-1.4,-1.4) rectangle (0,0)
	(0,0) -- (1.4,-0.6) -- (0,-1.4)
	(0,0) -- (-0.6,1.4) -- (-1.4,0)
	(1.4,-0.6) -- (1.4,0.6) -- (0.6,1.4) -- (-0.6,1.4)
	;
			
\draw[very thick, ->]
	(1.8,0) -- ++(0.8,0);

\node at (1.1,-1.3) {$T$};


\begin{scope}[xshift=4.4cm]

\draw
	(-1.4,-0.7) -- (0.5,-0.7) -- (0.7,-0.3) -- (0.5,0.5)
	(0.5,-0.7) -- (0.7,-1)
	(-0.7,-1.4) -- (-0.7,0.5) -- (-0.3,0.7) -- (0.5,0.5)
	(-0.7,0.5) -- (-1,0.7)
	(0,1.4) -- (0.5,0.5) -- (1.4,0)
	(1,1) -- (0.5,0.5)
	;

\fill
	(0,0) circle (0.1)
	(0,-1.4) circle (0.1)
	(-1.4,0) circle (0.1)
	(-1.4,-1.4) circle (0.1)
	(-0.6,1.4) circle (0.1)
	(0.6,1.4) circle (0.1)
	(1.4,0.6) circle (0.1)
	(1.4,-0.6) circle (0.1);	

\filldraw[fill=white]
	(-0.7,0.5) circle (0.1)
	(-0.7,-0.7) circle (0.1)
	(0.5,-0.7) circle (0.1)
	(0.5,0.5) circle (0.1);

\node at (1.1,-1.3) {$T(4)$};
	
\end{scope}

\end{tikzpicture}
\caption{Quadrilateral subdivision.}
\label{4div}
\end{figure}
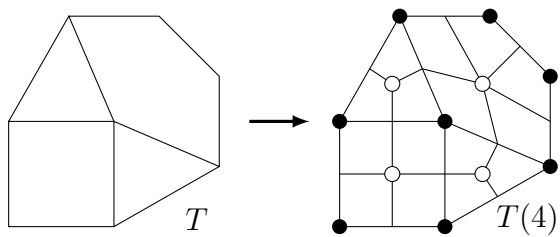

There are two types of tiles in a quadrilateral subdivision. The first is the quadrilateral $Q$ given by the first of Figure \ref{4div_tile}. It is non-degenerate, with a $\bullet$-vertex (from the original tiling $T$), a $\circ$-vertex (from the center of a tile in $T$), and two vertices of degree $4$, such that the $\bullet$-vertex and $\circ$-vertex are not adjacent. The second quadrilateral $Q'$ is degenerate, obtained by identifying the two degree $4$ vertices in the first of Figure \ref{4div_tile} in twisted way, and therefore gives a M\"obius band. This happens when two adjacent twisted edges of a tile in $T$ are identified as in the second of Figure \ref{4div_tile} (also see Figure \ref{degenerate}). We denote the quadrilateral by $Q'$. This forces two middle points labeled by small $\bullet$ to be identified, and the identified vertex has degree $4$. Together with the two nearby quadrilaterals $\circled{1}$ and $\circled{2}$, the three quadrilaterals form the M\"obius band in the third of Figure \ref{4div_tile}. We note that more identifications along the boundary of the M\"obius band may happen when vertices or edges of $\circled{1}$ and $\circled{2}$ are identified.

\begin{figure}[htp]
\centering
\begin{tikzpicture}[>=latex]


\draw
	(-0.6,-0.6) rectangle (0.6,0.6);

\fill
	(0.6,0.6) circle (0.1);
\filldraw[fill=white]
	(-0.6,-0.6) circle (0.1);

\node at (0.7,-0.7) {\small 4};
\node at (-0.7,0.7) {\small 4};

\node at (0,0) {$Q$};


\begin{scope}[xshift=3.3cm]

\fill[gray!30]	
	(0,0.8) -- (-0.6,0.5) -- (0,-0.2) -- (0.6,0.5) -- cycle;

\draw
	(0,0.8) -- (1.2,0.2) -- (0.8,-0.8)
	(0,0.8) -- (-1.2,0.2) -- (-0.8,-0.8)
	(-0.6,0.5) -- (0,-0.2) -- (0.6,0.5)
	(-1,-0.3) -- (0,-0.2) -- (1,-0.3);

\fill
	(0,0.8) circle (0.1)
	(1.2,0.2) circle (0.1)
	(-1.2,0.2) circle (0.1);
\filldraw[fill=white]
	(0,-0.2) circle (0.1);

\fill
	(0.6,0.5) circle (0.04)
	(-0.6,0.5) circle (0.04);

\filldraw[fill=white]
	(1,-0.3) circle (0.04)
	(-1,-0.3) circle (0.04);
		
\foreach \a in {-1,1}
{
\draw[->]
	(0.3*\a,0.65) -- ++(0.1,-0.05*\a);
\draw[->]
	(0.9*\a,0.35) -- ++(0.1,-0.05*\a);
}

\node at (0.4,0.8) {\scriptsize $a_1$};
\node at (-0.4,0.8) {\scriptsize $a_0$};
\node at (1,0.48) {\scriptsize $a_0$};
\node at (-0.9,0.5) {\scriptsize $a_1$};

\node[draw, shape=circle, inner sep=0.5] at (0.7,0.1) {\small $1$};
\node[draw, shape=circle, inner sep=0.5] at (-0.7,0.1) {\small $2$};

\node at (0,0.3) {$Q'$};

\end{scope}


\begin{scope}[xshift=7cm] 

\fill[gray!30]
	(-1.2,0.6) to[out=-80,in=180] 
	(0.0,-0.6) to[out=-0,in=-90] 
	(0.9,0) to[out=-40,in=0]
	(0,-0.8) to[out=180,in=-10] 
	(-1.2,-0.6)
	(-1.2,0.6) to[out=-20,in=90] 
	(0.9,0) to[out=140,in=50] 
	(-0.95,-0.05) to[out=230,in=70] 
	(-1.2,-0.6);

\draw
	(-1.2,0.6) to[out=30,in=90] 
	(1.4,0) to[out=-90,in=-30] 
	(-1.2,-0.6)
	(-1.2,0.6) to[out=-50,in=-90] 
	(0.4,0) to[out=90,in=30] 
	(-0.5,0)
	(-1.2,0.6) to[out=-80,in=180] 
	(0.0,-0.6) to[out=-0,in=-90] 
	(0.9,0) to[out=90,in=-20] 
	(-1.2,0.6)
	(0.9,0) to[out=-40,in=0]
	(0,-0.8) to[out=180,in=-10] 
	(-1.2,-0.6)
	(0.9,0) to[out=140,in=50] 
	(-0.95,-0.05);	

\draw[dashed]
	(-0.5,0) to[out=210,in=50] (-1.2,-0.6)
	(-1.2,-0.6) -- (-1.2,0.6);	

\draw[dashed]
	(-0.95,-0.05) to[out=230,in=70] 
	(-1.2,-0.6);
				
\fill
	(-1.2,0.6) circle (0.1);
\filldraw[fill=white]
	(-1.2,-0.6) circle (0.1);

\draw[->]
	(0.0,-0.6) -- ++(0.1,0);
\draw[->]
	(0.0,0.51) -- ++(-0.1,0);

\fill
	(0.9,0) circle (0.04);

\filldraw[fill=white]
	(0.4,0) circle (0.04)
	(1.4,0) circle (0.04);
		
\node at (0.0,-0.45) {\scriptsize $a_1$};
\node at (0.0,0.65) {\scriptsize $a_0$};

\node[draw, shape=circle, inner sep=0.5] at (1.1,0.3) {\small $1$};
\node[draw, shape=circle, inner sep=0.5] at (0.6,-0.2) {\small $2$};

\end{scope}

\end{tikzpicture}
\caption{Tiles in quadrilateral subdivision.}
\label{4div_tile}
\end{figure}
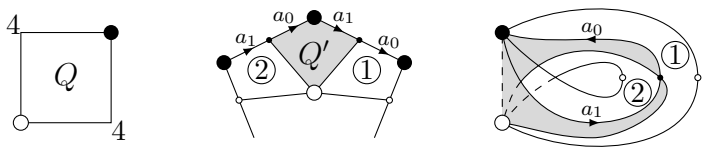

The following is the analogue of Theorem \ref{ps_thm1}.

\begin{theorem}\label{dps_thm1}
A quadrilateral tiling of a surface is the quadrilateral subdivision of a tiling, if and only if it is possible to label some vertices by $\bullet$ and some other vertices by $\circ$, such that every tile is $Q$ or $Q'$. 
\end{theorem}

\begin{proof}
Figure \ref{proof2} shows three consecutive tiles $\circled{1},\circled{2},\circled{3}$ around a $\circ$-vertex. We have the $\bullet$-vertices of these three tiles, and indicate the non-labeled degree $4$ vertices of $\circled{2}$ by small $\bullet$. We also indicate four quadrilateral tiles around each small $\bullet$ vertex of degree $4$.  

The tile $\circled{2}$ is $Q$ or $Q'$. If $\circled{2}$ is the non-degenerate $Q$, then there is no identification among four edges $x,y,u,v$ (although $\bullet$-vertices of $\circled{1},\circled{2},\circled{3}$ may be identified). If $\circled{2}$ is the degenerate $Q'$, we need to argue that we are in the situation described by the second and third of Figure \ref{4div_tile}. Since the identification of the two small $\bullet$ is twisted, and the identified vertex has degree $4$, we must have $\circled{2}$ identified with $\circled{2'}$, such that $x$ and $u$ are identified. This also implies that $y$ is identified with $v$, and $\circled{3}$ identified with $\circled{3'}$. Then the edge $u+v$ is identified with $x+y$, and we get the situation described by the second and third of Figure \ref{4div_tile}.

\begin{figure}[h]
\centering
\begin{tikzpicture}[>=latex,scale=1]

\draw	
	(0,-0.8) -- (1.6,-0.8)
	(-0.8,0) -- (1.6,0) 
	(-0.8,0.8) -- (1.6,0.8) 
	(-0.8,1.6) -- (0.8,1.6) 
	(-0.8,0) -- (-0.8,1.6)
	(0,-0.8) -- (0,1.6)
	(0.8,-0.8) -- (0.8,1.6)
	(1.6,-0.8) -- (1.6,0.8);
	
\fill
	(0.8,-0.8) circle (0.1)
	(0.8,0.8) circle (0.1)
	(-0.8,0.8) circle (0.1);

\filldraw[fill=white]
	(0,0) circle (0.1)
	(1.6,0) circle (0.1)
	(0,1.6) circle (0.1);

\fill
	(0.8,0) circle (0.04)
	(0,0.8) circle (0.04);
	
\node[draw, shape=circle, inner sep=0.5] at (0.4,-0.4) {\small $1$};
\node[draw, shape=circle, inner sep=0.5] at (0.4,0.4) {\small $2$};
\node[draw, shape=circle, inner sep=0.5] at (-0.4,0.4) {\small $3$};

\node[draw, shape=circle, inner sep=0] at (1.3,-0.4) {\small $2'$};
\node[draw, shape=circle, inner sep=0] at (1.3,0.4) {\small $3'$};

\node at (0.92,-0.4) {\scriptsize $u$};
\node at (0.92,0.4) {\scriptsize $v$};
\node at (0.4,0.92) {\scriptsize $x$};
\node at (-0.4,0.92) {\scriptsize $y$};

\end{tikzpicture}
\caption{Construction of quadrilateral subdivision.}
\label{proof2}
\end{figure}
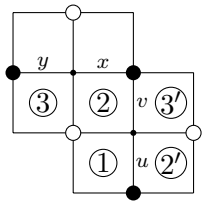

The union of the quadrilateral tiles around the initial $\circ$-vertex is a polygon. The polygon may be degenerate, but the discussion above shows that it cannot have the forbidden degeneracy in Figure \ref{degenerate}. Moreover, by comparing with the nearby $\circ$-vertex shared by $\circled{2'}$ and $\circled{3'}$, we find the edge of the polygon around initial $\circ$-vertex is locally shared with a polygon around another $\circ$-vertex (which might be the initial polygon by edge identification). We find that all the polygons form a tiling, and the quadrilateral tiling is the quadrilateral subdivision of this tiling.
\end{proof}

The following gives a criterion for double pentagonal subdivision.

\begin{theorem}\label{dps_thm2}
The quadrilateral subdivision of a tiling is (pentagonally) subdivisible if and only if the surface is orientable.
\end{theorem}

\begin{proof}
Given a tiling $T$ of a surface $S$, the fundamental cycles of $S$ can be represented by cycles $C$ of $T$. Then the fundamental cycles of $S$ are represented by the same cycles $C$ in the quadrilateral subdivision $T(4)$, except that every edge has been subdivided into two edges. In particular, the cycles of $T(4)$ representing fundamental cycles have even numbers of edges. By Proposition \ref{quad_theroem2}, the quadrilateral tiling $T(4)$ is subdivisible if and only if the surface $S$ is orientable.
\end{proof}

\end{document}